%% file: artigo_unificado.tex
\documentclass[11pt,a4paper]{amsart}
\usepackage[utf8]{inputenc}
\usepackage[T1]{fontenc}
\usepackage{amsmath,amssymb,amsthm,mathtools}
\usepackage{booktabs}
\usepackage{enumitem}
\usepackage{seqsplit}
\usepackage{graphicx}
\usepackage[margin=2.6cm]{geometry}
\usepackage{xcolor}
\usepackage{hyperref}
\hypersetup{colorlinks=true,linkcolor=blue!55!black,citecolor=teal!60!black,urlcolor=purple!60!black}

\newcommand{\Q}{\mathbb{Q}}
\newcommand{\R}{\mathbb{R}}

\newcommand{\Z}{\mathbb{Z}}
\DeclareMathOperator{\Gal}{Gal}
\DeclareMathOperator{\Res}{Res}
\DeclareMathOperator{\supp}{supp}

\newtheorem{teorema}{Theorem}
\newtheorem{proposicao}[teorema]{Proposition}

\newtheorem{lema}[teorema]{Lemma}
\theoremstyle{definition}
\newtheorem{definicao}[teorema]{Definition}
\newtheorem{observacao}[teorema]{Remark}
\newtheorem{algoritmo}[teorema]{Algorithm}

\title[The Exoo--Ismailescu pair: origami ($12T236$) vs.\ non-solvable ($S_{20}$)]{Exact
certification of the coordinate fields of the triangle-free Exoo--Ismailescu
unit-distance graphs $EI_{17}$ and $EI_{19}$ (\textsc{HoG} 51375, 51376): a
solvable--non-solvable dichotomy (origami vs.\ $S_{20}$) and the Laman-number conjecture}

\author{Haroldo Costa Silva Filho}
\address{CEFET-RJ; PGCComp --- Programa de P\'os-Gradua\c{c}\~ao em Ci\^encias
Computacionais, Universidade do Estado do Rio de Janeiro (UERJ), Rio de Janeiro,
Brazil}
\email{haroldosfilho@gmail.com}
\date{\today}

\begin{document}

\begin{abstract}
We certify, exactly, the coordinate fields of a faithful planar realization of \emph{two}
neighbouring triangle-free Exoo--Ismailescu unit-distance graphs (UDGs), and show they
realize the two \emph{opposite} extremes of the constructibility hierarchy. The
$17$-vertex graph $EI_{17}$ (\textsc{House of Graphs} 51375) is the \emph{smallest}
triangle-free UDG with chromatic number~$4$ \cite{Soifer}; the $19$-vertex graph
$EI_{19}$ (\textsc{HoG} 51376) is its state-of-the-art origami neighbour. In both, fixing
a rational base edge, the remaining vertices are intersections of unit circles --- each
on the \emph{radical axis} of its two neighbours, a tower of square roots over the free
angles --- and a small closure system locks the realization. For $EI_{19}$ the base lies
in $\Q(\sqrt2,\sqrt5,\sqrt7)$ and a single free angle has an \emph{irreducible} degree
$12=2^2\cdot3$ minimal polynomial with Galois group the \emph{solvable} transitive group
$12T236$ (order $2304=2^8\cdot3^2$): not ruler-and-compass, but \emph{origami}-constructible
(the cubic Beloch fold $O6$ necessary, in \emph{casus irreducibilis}). For $EI_{17}$ two
free angles $\theta_4,\theta_5$ are locked by two closures, whose resultant
$\Res_{\theta_5}(g_1,g_2)$ is \emph{irreducible} of degree $\mathbf{20}=2^2\cdot5$ with
Galois group the \emph{full symmetric group} $S_{20}$ (a Frobenius census exhibits a
$17$-cycle, forcing $A_{20}$ by Jordan, and an odd $20$-cycle, raising it to $S_{20}$):
\emph{non-solvable}, so the coordinates are \emph{not expressible by radicals} --- neither
compass nor origami of any fold order. Thus the smallest triangle-free $4$-chromatic UDG
is the \emph{generic}, maximally exotic case, the exact opposite of its origami neighbour.
We give the full certification pipeline as explicit algorithms, record two methodological
pitfalls, and read the pair through a conjectural bridge between the \emph{Laman number}
and the Galois group.
\end{abstract}

\maketitle

\section{Introduction}
A \emph{unit-distance graph} (UDG) in the plane is a graph $G=(V,E)$ with a map
$p:V\to\R^2$ such that $\lVert p(u)-p(v)\rVert=1$ whenever $uv\in E$; the realization is
\emph{faithful} (strict) if moreover $\lVert p(u)-p(v)\rVert\neq1$ for every
$uv\notin E$, so the edges are \emph{exactly} the unit-distance pairs. UDGs lie at the
heart of the Hadwiger--Nelson problem on the chromatic number of the plane,
$\chi(\R^2)\in\{5,6,7\}$ \cite{deGrey,Soifer}, whose lower bound was raised to $5$ by de
Grey \cite{deGrey} and revisited by Exoo and Ismailescu \cite{ExooIsmailescu}.

\begin{observacao}[Extremal geometry of unit distances]\label{obs:extremal}
Two extremal problems frame this graph. The \emph{Hadwiger--Nelson} problem asks the
least number of colours so that no two points at distance $1$ share a colour; finite
UDGs of large chromatic number, such as the Exoo--Ismailescu family, drive its lower
bound \cite{deGrey,ExooIsmailescu,Soifer}. The \emph{Erdős unit-distance problem} (1946)
asks the maximum number $u(n)$ of unit-distance pairs among $n$ planar points
\cite{Erdos1946}; Erdős conjectured $u(n)=n^{1+o(1)}$, a belief only recently overturned
by an explicit family with $n^{1+\delta}$ pairs for a fixed $\delta>0$ \cite{ErdosUD}. In
both problems the objects are small, rigid, triangle-free UDGs; determining their
coordinate fields exactly --- the aim here --- is the arithmetic counterpart of these
extremal questions.
\end{observacao}

The $17$-vertex Exoo--Ismailescu graph \cite{ExooIsmailescu}, catalogued as
\textsc{HoG}~51375, is the \emph{smallest} triangle-free planar UDG with chromatic
number~$4$ (Soifer \cite{Soifer}). It has $|E|=31=2n-3$ edges with $n=17$, hence is
\emph{isostatic} (minimally rigid in the sense of Laman \cite{Laman,GSS}); it is
triangle-free of girth~$4$, with degree sequence six $3$'s and eleven $4$'s and
automorphism group of order~$2$. Its $19$-vertex sibling \textsc{HoG}~51376 is treated
in the companion paper \cite{ei19}, where its field is shown origami-constructible
(degree $12$, Galois group $12T236$, solvable).

For $51375$ we answer the arithmetic--geometric question --- in which number field do
the coordinates of a faithful realization live, and how constructible is that field? ---
and the answer (Theorem~\ref{teo:principal}) is the \emph{opposite} of the $19$-vertex
case: the field is a degree-$20$ extension whose Galois group is the full symmetric
group $S_{20}$, hence \emph{not solvable by radicals}. This answers, for this graph, a
question of Edward Pegg~Jr.\ on the coordinate fields of faithful realizations of small
Laman and unit-distance graphs, and exhibits the smallest triangle-free $4$-chromatic
UDG as a maximally non-constructible instance.

\paragraph{Edge list ($31$).}
$(1,4)$, $(1,7)$, $(1,8)$, $(2,3)$, $(2,6)$, $(2,10)$, $(3,5)$, $(3,7)$, $(3,12)$,
$(4,6)$, $(4,13)$, $(4,14)$, $(5,8)$, $(5,9)$, $(5,14)$, $(6,11)$, $(6,16)$, $(7,9)$,
$(7,15)$, $(8,13)$, $(8,15)$, $(9,10)$, $(9,11)$, $(10,12)$, $(10,16)$, $(11,12)$,
$(11,17)$, $(13,17)$, $(14,15)$, $(14,17)$, $(16,17)$.

\begin{figure}[htbp]
\centering
\includegraphics[width=.7\textwidth]{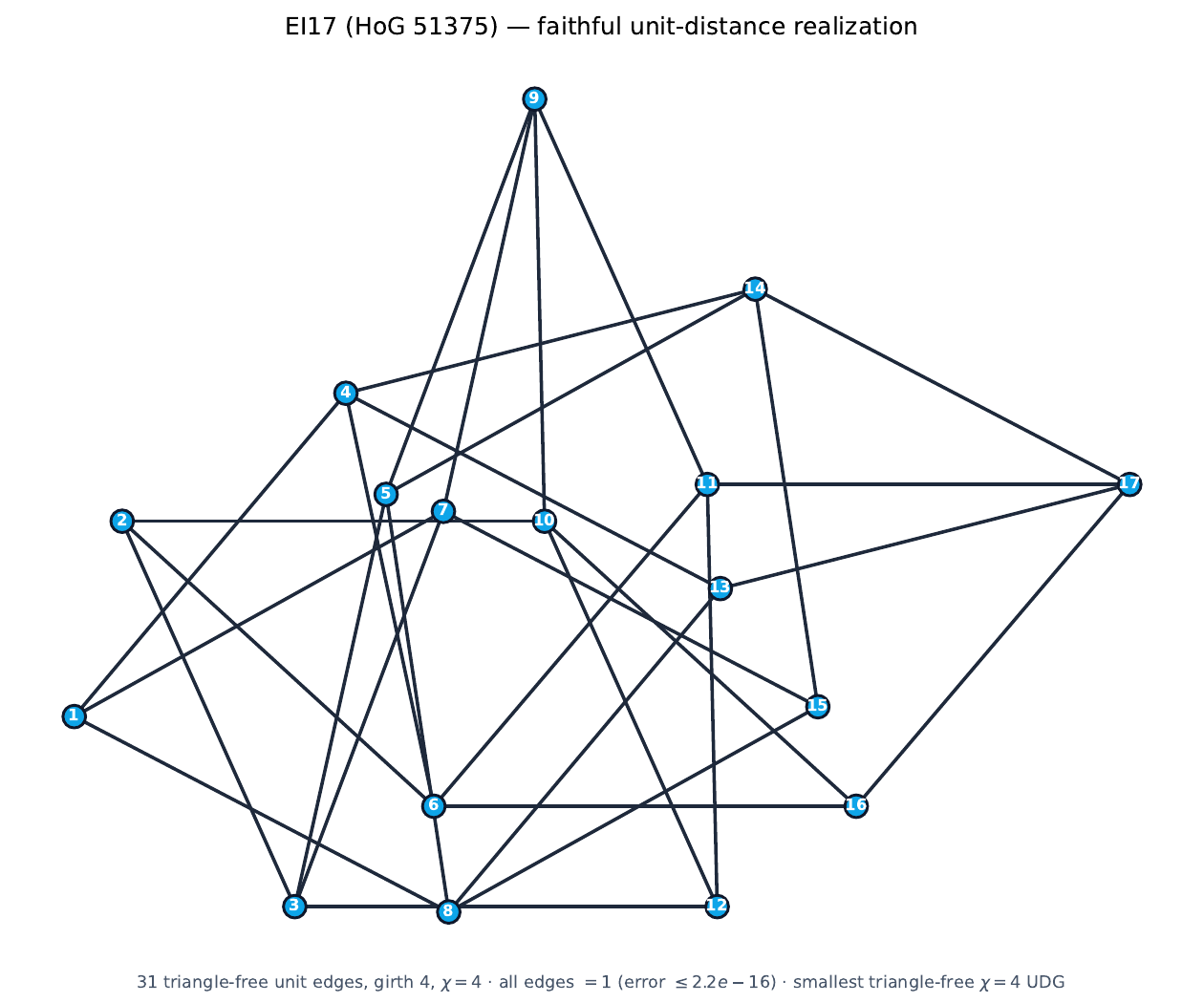}
\caption{The certified faithful realization of $EI_{17}$ (\textsc{HoG} 51375): $31$
triangle-free unit edges, girth~$4$, $\chi=4$. All edges equal $1$ to
$e_{\max}\le2\cdot10^{-16}$ (double precision) and $<10^{-400}$ after Newton refinement;
no non-edge is at distance~$1$ (strict).}
\label{fig:real}
\end{figure}

\section{Preliminaries}\label{sec:prelim}

\subsection{Rigidity, isostaticity, realizations}
The \emph{rigidity matrix} of a realization $p$ of $G=(V,E)$ is the $|E|\times2|V|$
matrix whose row for $uv\in E$ has $p(u)-p(v)$ in the columns of $u$ and $p(v)-p(u)$ in
those of $v$ \cite{GSS,Laman}. A realization is \emph{infinitesimally rigid} when this
matrix has rank $2|V|-3$; a graph is \emph{generically rigid} (Laman) iff it has a
spanning subgraph with $2|V|-3$ edges no subset $F$ of which exceeds $2|V(F)|-3$
\cite{Laman}. Here $|E|=2|V|-3$, so $G$ is \emph{isostatic}: at a regular realization
the only infinitesimal motions are the plane isometries. Fixing a gauge (an edge) makes
the realization \emph{isolated} --- a point of a $0$-dimensional variety of the
unit-distance equations --- which is what gives meaning to \emph{the coordinate field}.

\begin{definicao}[Coordinate field]\label{def:field}
For a faithful realization $p$ of an isostatic $G$ with a fixed rational gauge, the
\emph{coordinate field} is $K=\Q\bigl(\{p(v):v\in V\}\bigr)\subseteq\R$. When $G$ is not
globally rigid \cite{Connelly} there are several real realizations, possibly in distinct
fields; $K$ is that of the chosen (here, rational-based) realization.
\end{definicao}
\subsection{Maehara's theorem: rigid distances are algebraic}\label{ssec:maehara}
That $K$ (Definition~\ref{def:field}) is a genuine \emph{number field} rests on a theorem
of Maehara \cite{Maehara}. Call a graph drawn in the plane \emph{flexible} if some
nontrivial continuous motion of its vertices preserves all edge lengths, and
\emph{rigid} otherwise; an isostatic graph at a regular realization is rigid. Let
$\Gamma$ be the set of Euclidean distances occurring between two vertices of some rigid
\emph{unit-distance} graph, and $A^{+}$ the set of positive algebraic numbers.

\begin{teorema}[Maehara \cite{Maehara}]\label{teo:maehara}
$\Gamma=A^{+}$: a real $d>0$ is the distance between two vertices of some rigid
unit-distance graph in the plane if and only if $d$ is algebraic.
\end{teorema}

The forward inclusion $\Gamma\subseteq A^{+}$ is the point we use: once a gauge is fixed,
the vertices of a rigid graph are \emph{isolated} solutions of a polynomial system with
rational coefficients, so their coordinates --- and all mutual distances --- are
algebraic (Maehara derives it from a theorem on real-algebraic critical values
\cite{Maehara}). The converse $A^{+}\subseteq\Gamma$ is \emph{constructive}: Maehara
realizes any algebraic $d$ by assembling classical Kempe linkages --- the
\emph{antiparallelogram}, the \emph{reverser} $R(O;A,B,C)$ (four rods forming two
antiparallelograms, enforcing $\angle AOB=\angle COB$), and the \emph{$n$-fan} built from
reversers --- together with gadgets that produce $nd$ and $c^{2}/d$ from given rigid
distances. This step-by-step assembly of rigid pieces is the arithmetic ancestor of the
circle-intersection construction of \S\ref{sec:constr} and of the universality theorems
of Mnëv and Kapovich--Millson \cite{Mnev,KapovichMillson} behind the $\exists\R$ hardness
of \S\ref{sec:er}. Theorem~\ref{teo:maehara} is an \emph{existence/universality}
statement (\emph{every} algebraic number occurs as some rigid-UDG distance); the present
paper solves the exact \emph{converse} for one fixed graph --- determining \emph{which}
algebraic numbers are the coordinates of $EI_{17}$, namely a degree-$20$ field with
Galois group $S_{20}$.

\subsection{The number of realizations (Laman number)}
Over a generic choice of edge lengths $\lambda$, the complex solutions of the
distance equations (modulo isometry) form a finite set; its cardinality $N(G)$ is the
\emph{Laman number}, a combinatorial invariant computable by the recursion of
Capco--Gallet--Grasegger--Koutschan--Lubbes--Schicho \cite{Capco} and bounded above by
the mixed volume (BKK) of the system. Over the generic base $K_0=\Q(\lambda)$ the field
of one realization has degree exactly $N(G)$ over $K_0$, and the monodromy of the family
--- equivalently the Galois group --- acts transitively on the $N(G)$ realizations
\cite{Capco,GKKPS}.

\subsection{Constructibility: compass, origami, and beyond}
\begin{definicao}[Constructibility tiers]\label{def:tiers}
A real number $\alpha$ is:
\begin{enumerate}[leftmargin=2em,itemsep=1pt]
\item \emph{ruler-and-compass constructible} iff it lies in a tower of \emph{quadratic}
extensions of $\Q$; equivalently $[\Q(\alpha):\Q]$ is a power of $2$ and the Galois group
of its normal closure is a $2$-group (Gauss--Wantzel) \cite{Cox};
\item \emph{origami-constructible} (single-fold, Huzita--Hatori axioms, with the cubic
Beloch fold $O6$) iff it lies in a $\{2,3\}$-tower --- a Pierpont number, Galois group a
$\{2,3\}$-group \cite{Alperin,Huzita,Hull,OllerMarcen}; the mathematics of paper folding
(\emph{origametry}) is developed systematically by Hull \cite{Hull};
\item \emph{expressible by radicals} iff the Galois group of its minimal polynomial is
\emph{solvable} (Galois) \cite{Cox}.
\end{enumerate}
A field none of whose generators is radical-expressible we call \emph{exotic}. The tier
is read off the prime support of the degree/group order: $\{2\}$ (compass), $\{2,3\}$
(origami), or a prime $\ge5$ (exotic).
\end{definicao}
Multifold origami (Alperin--Lang) extends single-fold origami but still produces only
radical (indeed, for bounded fold number, solvable) numbers; a non-solvable field lies
beyond all of these.

\subsection{Permutation groups, solvability, Jordan's theorem}\label{ssec:groups}
We recall the group theory used in the Galois classification
\cite{DixonMortimer,Wielandt,Cox}.
\begin{definicao}\label{def:groups}
The \emph{symmetric group} $S_n$ is the group of all permutations of $\{1,\dots,n\}$
(order $n!$); the \emph{alternating group} $A_n\le S_n$ is the subgroup of \emph{even}
permutations --- products of an even number of transpositions (order $n!/2$). A cycle of
length $\ell$ is even iff $\ell$ is odd. A group $G$ is \emph{solvable} if it admits a
subnormal series $\{e\}=G_0\triangleleft\cdots\triangleleft G_r=G$ with all quotients
$G_{i+1}/G_i$ \emph{abelian} (equivalently, its derived series reaches $\{e\}$). A
non-abelian group with no proper nontrivial normal subgroup is \emph{simple}. For
$n\ge5$, $A_n$ is simple and non-abelian, so $S_n$ is \emph{not} solvable.
\end{definicao}
A subgroup $G\le S_n$ is \emph{transitive} if it has one orbit and \emph{primitive} if in
addition it preserves no nontrivial partition (the point stabiliser is then maximal).
For the Galois group acting on the roots of an irreducible polynomial, transitivity is
automatic, and primitivity is equivalent to the field having \emph{no proper subfield}.

\begin{teorema}[Jordan; see \cite{DixonMortimer,Wielandt}]\label{teo:jordan}
A primitive subgroup of $S_n$ containing a $p$-cycle for a prime $p\le n-3$ contains
$A_n$.
\end{teorema}

By the Frobenius/Chebotarev density theorem, factoring an irreducible integer polynomial
modulo the primes realises, with the correct densities, the cycle types of the Galois
group acting on the roots \cite{Cohen,Cox}; this is what makes
Algorithm~\ref{alg:gal} effective, and the monodromy of the realization family (as edge
lengths vary) is precisely this Galois action on the $N$ realizations \cite{Capco,GKKPS}.

\section{The Euclidean construction: chords on the radical axis}\label{sec:constr}
Fix the gauge $v_{17}=(0,0)$, $v_{11}=(-1,0)$ (a rational unit edge). Two vertices are
\emph{free}: $v_6=v_{11}+(\cos\theta_4,\sin\theta_4)$ and
$v_2=v_6+(\cos\theta_5,\sin\theta_5)$. Every other vertex $v$ is the intersection of two
\emph{unit} circles centred at already-placed neighbours $A,B$ (Figure~\ref{fig:circ}).

\begin{figure}[htbp]
\centering
\includegraphics[width=.72\textwidth]{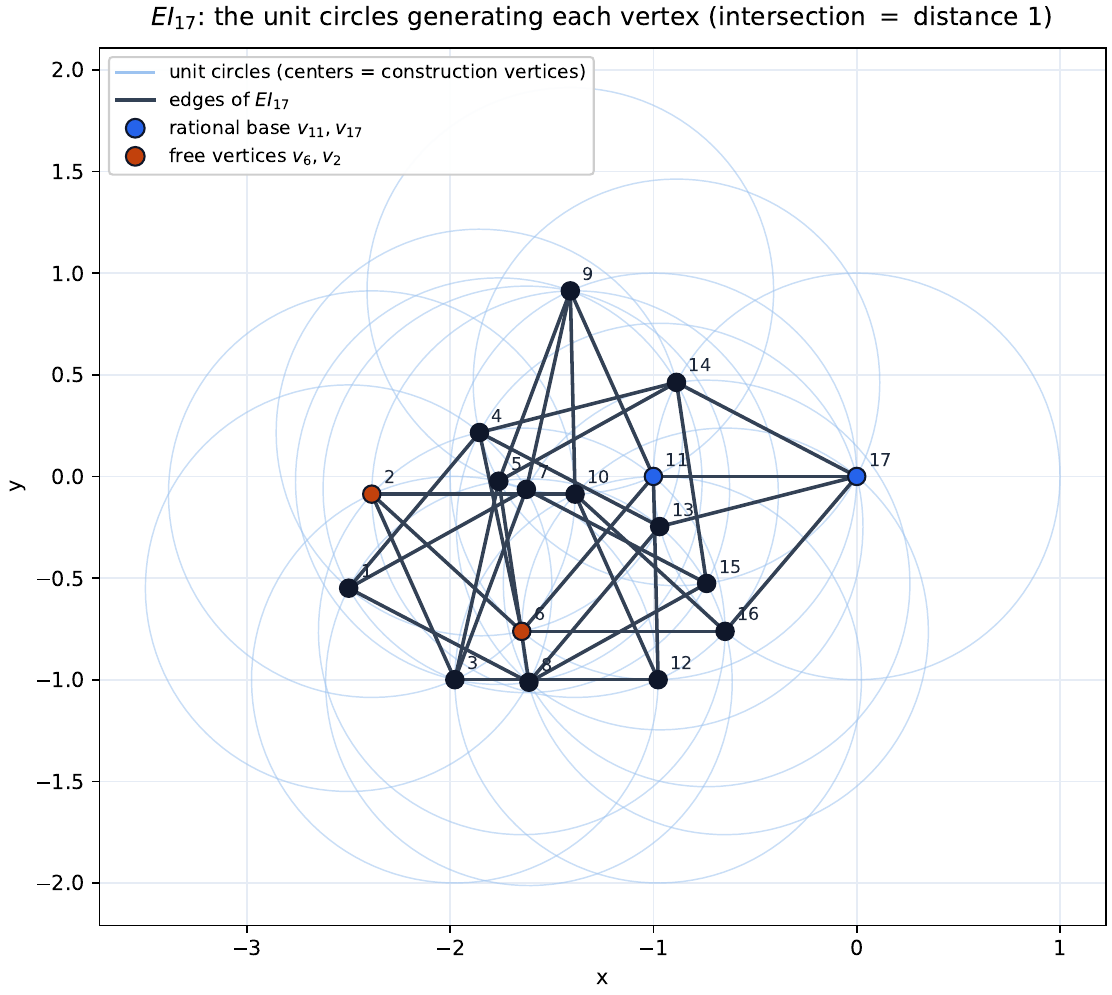}
\caption{The web of unit circles (centres at the construction vertices) whose pairwise
intersections are the vertices of $EI_{17}$; each edge joins two vertices whose circles
pass through each other.}
\label{fig:circ}
\end{figure}

\begin{proposicao}[Radical-axis certificate, Euclidean]\label{prop:radaxis}
Let $A,B$ be centres of unit circles with $0<|A-B|<2$. Their \emph{radical axis} --- the
locus of equal power --- is, by equal radii, the perpendicular bisector of $AB$; it
meets each circle in the two points
\begin{equation}\label{eq:radaxis}
v \;=\; \frac{A+B}{2}\ \pm\ \sqrt{\,1-\tfrac{|A-B|^2}{4}\,}\ \cdot\ \frac{(B-A)^{\perp}}{|A-B|}.
\end{equation}
Each such point $v$ is the \emph{unique} solution, on its branch, of one \emph{linear}
and one \emph{quadratic} equation with known coefficients in $A,B$:
\begin{equation}\label{eq:refcert}
\underbrace{2(B-A)\cdot v=|B|^2-|A|^2}_{\text{radical axis (linear)}}
\qquad\text{and}\qquad
\underbrace{|v-A|^2=1}_{\text{unit chord (quadratic)}}.
\end{equation}
Consequently every vertex other than the rational base $v_{11},v_{17}$ is
\emph{exactly certified, referentially}, by \eqref{eq:refcert} in two earlier vertices,
and the whole realization is the unique solution of a \emph{triangular chain} of such
Euclidean conditions over $\Q$ --- a certificate needing no embedding into the
degree-$20$ field, with small-height coefficients.
\end{proposicao}
\begin{proof}
For circles $|P-A|^2=r_A^2$, $|P-B|^2=r_B^2$ the radical axis is
$\{P:|P-A|^2-|P-B|^2=r_A^2-r_B^2\}$; with $r_A=r_B=1$ this is
$2(B-A)\cdot P=|B|^2-|A|^2$, the perpendicular bisector of $AB$. A point of it with
$|P-A|=1$ lies on both circles, hence is an intersection; conversely the two
intersections have equal power, so lie on the radical axis. The two solutions of
\eqref{eq:refcert} are $v$ and its \emph{twin}; fixing the branch selects one. Each
construction vertex has two already-placed neighbours, so iterating in construction
order gives the triangular chain, and \eqref{eq:radaxis} solves \eqref{eq:refcert} in
closed form.
\end{proof}

\begin{figure}[htbp]
\centering
\includegraphics[width=.66\textwidth]{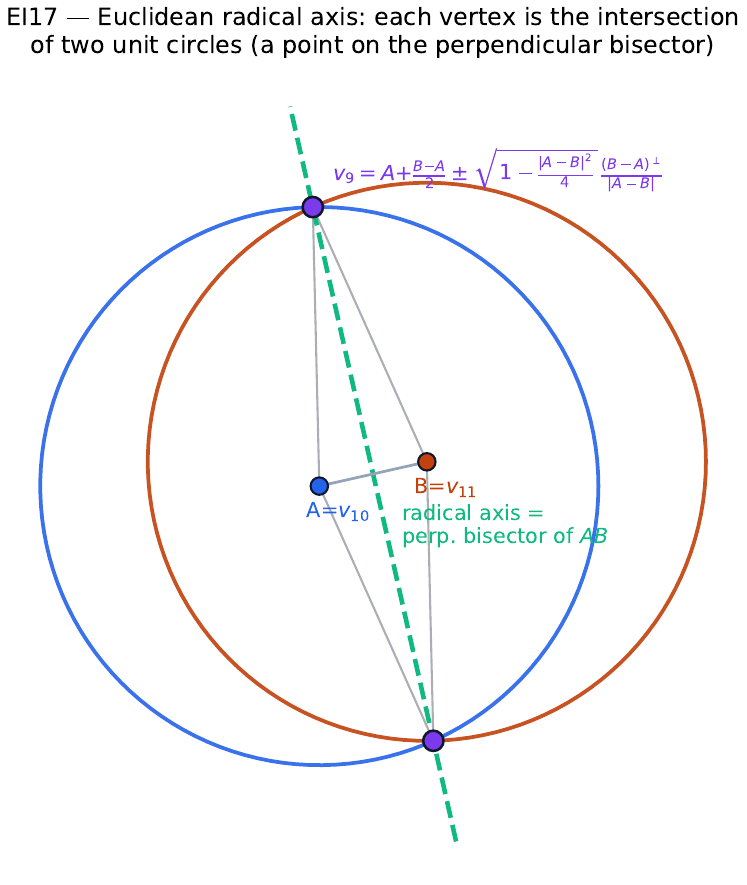}
\caption{The Euclidean content of Proposition~\ref{prop:radaxis}. Two unit circles with
centres $A,B$ meet on their radical axis --- the perpendicular bisector of $AB$ (equal
radii) --- at $v$ and its twin $v'$; the edges $Av,Bv$ are unit chords, and
$v=M\pm h\,(B-A)^{\perp}/|A-B|$ with $M=\tfrac{A+B}2$ and $h=\sqrt{1-|A-B|^2/4}$. Thus
each vertex is fixed by one linear (radical axis) and one quadratic (unit chord)
condition with known coefficients.}
\label{fig:radeuclid}
\end{figure}

Iterating \eqref{eq:radaxis} (Figure~\ref{fig:steps}), the fifteen non-base vertices
form a \emph{tower of square roots} over $\theta_4,\theta_5$: a ruler-and-compass
construction \emph{relative} to the two free angles. All non-radicality is therefore
concentrated in those two angles.

\begin{figure}[htbp]
\centering
\includegraphics[width=\textwidth]{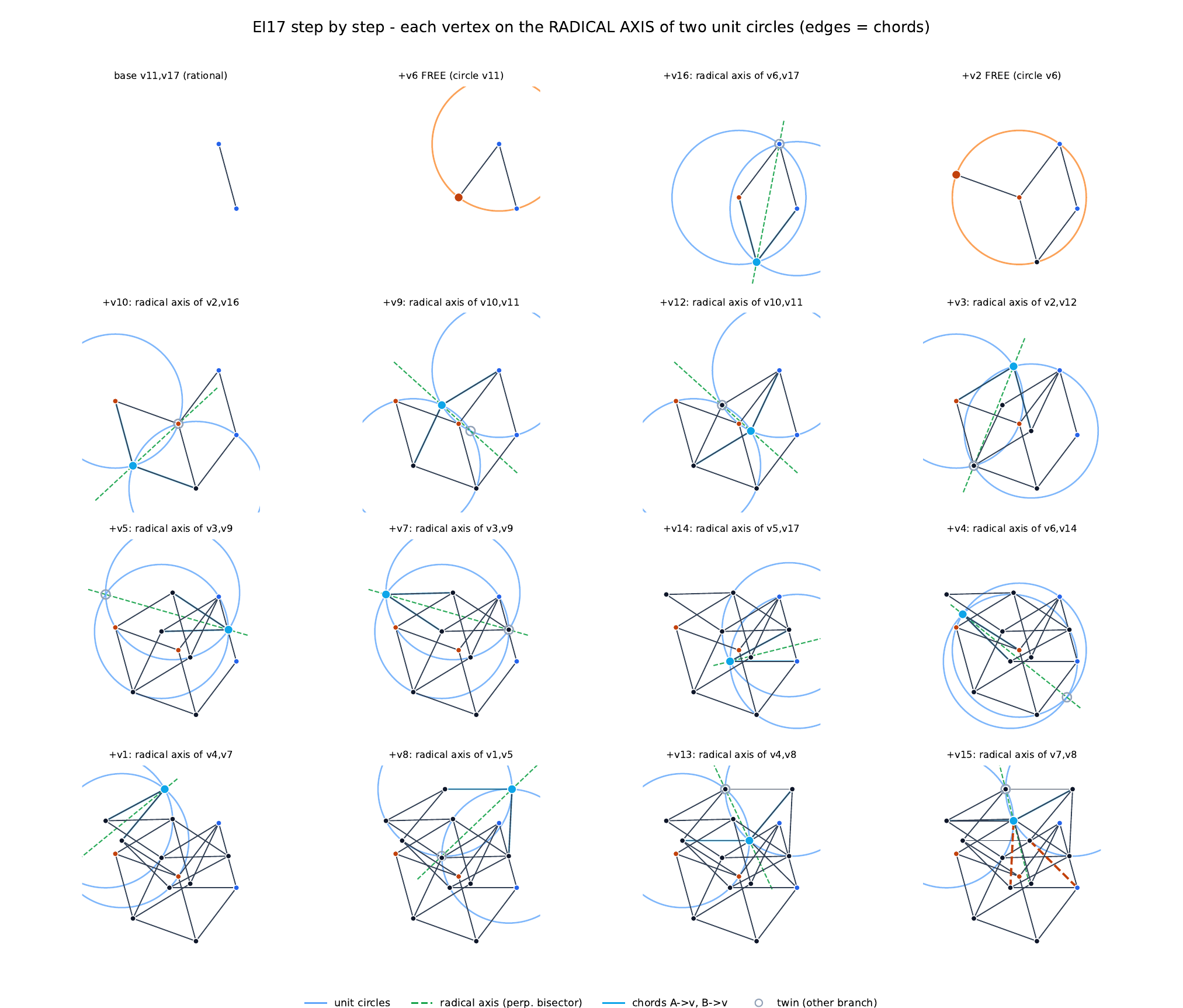}
\caption{Step-by-step construction of $EI_{17}$ showing the radical-axis geometry: each
tri-vertex (light blue) is born on the perpendicular bisector (dashed green) of two unit
circles, with its twin (grey); edges are the chords (blue). Base blue, free angles
$v_6,v_2$ orange; the last frame shows the two closures dashed orange.}
\label{fig:steps}
\end{figure}

\subsection{Unit rhombi}
The nine $4$-cycles of $EI_{17}$ are, in this realization, \emph{unit rhombi}
(side-$1$ parallelograms; Figure~\ref{fig:rhombi}); four have horizontal side $(1,0)$,
so that e.g.\ $v_2v_{10}v_{12}v_3$ closes. The graph is a braced quilt of rhombi, as in
the Moser spindle; the bracing is what makes the isostatic frame rigid.

\begin{figure}[htbp]
\centering
\includegraphics[width=.66\textwidth]{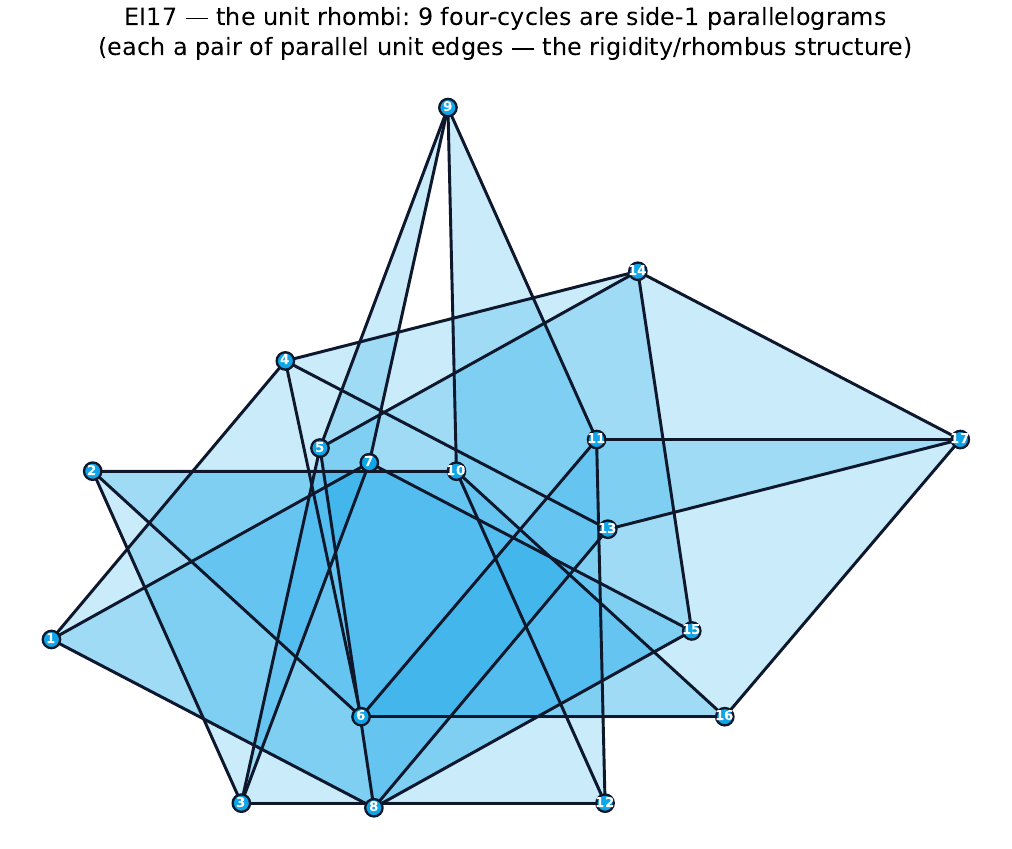}
\caption{The nine unit rhombi of $EI_{17}$ ($4$-cycles with all four sides equal to $1$).}
\label{fig:rhombi}
\end{figure}

\subsection{The closure system}
The two free angles are locked by the two remaining edges $(13,17)$ and $(14,15)$
(Figure~\ref{fig:mech}), giving two polynomial equations $g_1(\theta_4,\theta_5)=0$,
$g_2(\theta_4,\theta_5)=0$ --- two curves in the free-angle plane whose crossing is a
realization (Figure~\ref{fig:closure}). Eliminating $\theta_5$, the resultant
$p=\Res_{\theta_5}(g_1,g_2)$ is a one-variable polynomial whose degree is the number of
realizations.

\begin{figure}[htbp]
\centering
\includegraphics[width=.72\textwidth]{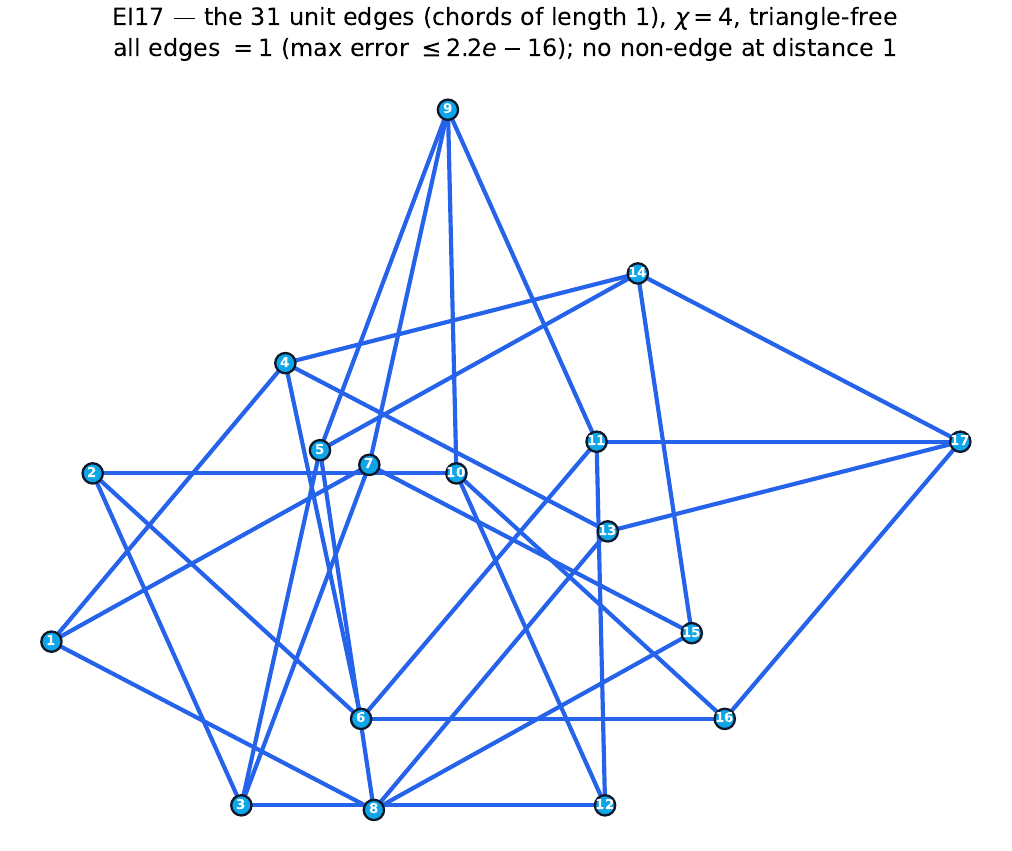}
\caption{Twenty-nine unit edges are produced for free by the circle intersections; the
two dashed \emph{closure} edges $(13,17),(14,15)$ hold only for special values of
$\theta_4,\theta_5$ and are what the degree-$20$ resultant governs.}
\label{fig:mech}
\end{figure}

\begin{figure}[htbp]
\centering
\includegraphics[width=.56\textwidth]{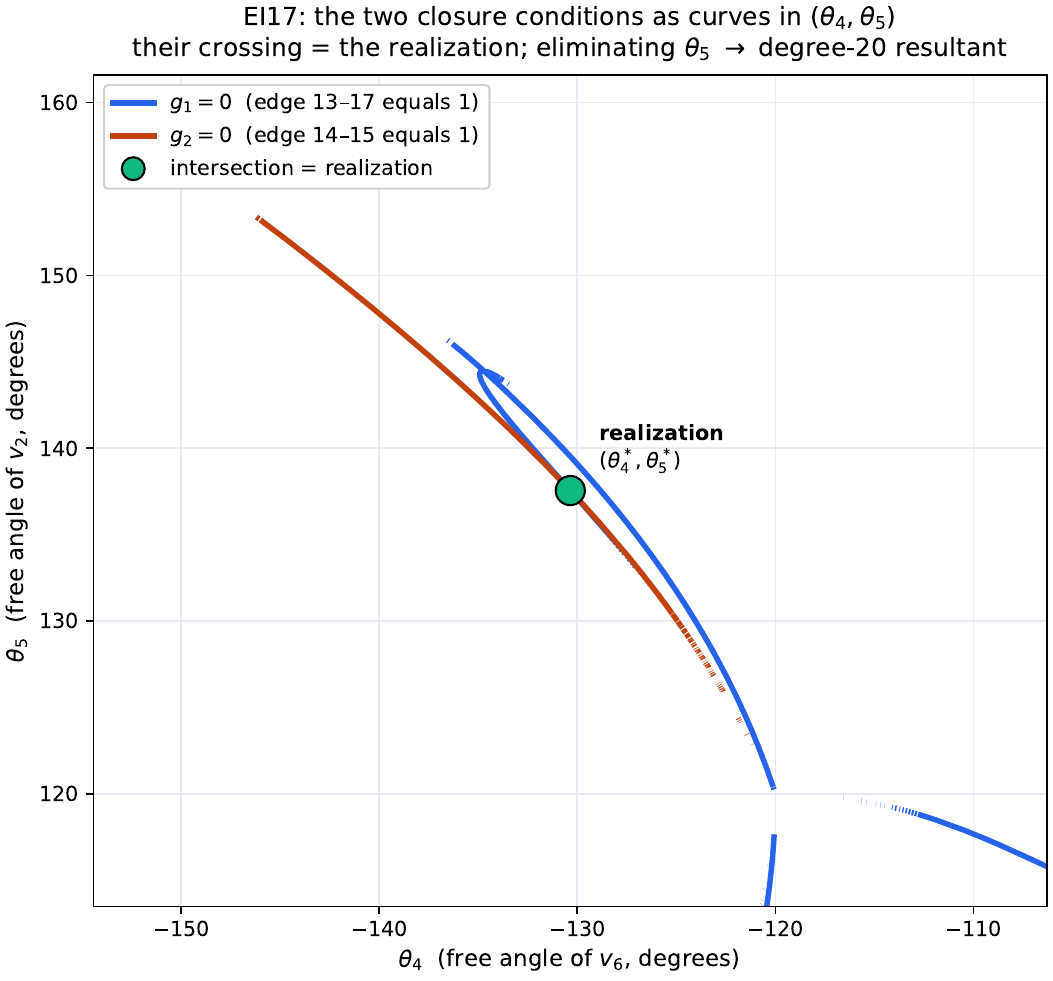}
\caption{The two closure conditions as curves in $(\theta_4,\theta_5)$: $g_1=0$ (blue,
edge $13$--$17$) and $g_2=0$ (orange, edge $14$--$15$); their crossing (green) is the
realization. Eliminating $\theta_5$ gives the degree-$20$ resultant.}
\label{fig:closure}
\end{figure}

\section{Certification methods}\label{sec:methods}
We describe the pipeline as explicit algorithms.

\begin{algoritmo}[High-precision realization]\label{alg:newton}
Gauge-fix $v_{17}=(0,0)$, $v_{11}=(-1,0)$; from the six-digit seed run Newton on the
$30$ edge residuals $|p(u)-p(v)|^2-1$ (excluding the base) until $\|F\|<10^{-\text{dps}}$,
in multiprecision (\texttt{mpmath}); certify by Gram--SVD: rank $2$, $e_{\max}<10^{-400}$,
minimal pairwise separation $>10^{-2}$, and zero non-edges at distance $1$.
\end{algoritmo}

\begin{algoritmo}[Integer-relation recognition, validated]\label{alg:pslq}
For a chosen coordinate $\alpha$, form the vector $(1,\alpha,\dots,\alpha^d)$ at $D$ digits
and run the \textsc{PSLQ} integer-relation algorithm (as \texttt{algdep})
\cite{FergusonBailey}; accept a candidate minimal polynomial only if it \emph{re-verifies}
at a higher precision $D'>D$ (residue validation). This guards against spurious
low-degree relations (Remark~\ref{obs:pitfall}).
\end{algoritmo}

\begin{algoritmo}[Continuity reduction]\label{alg:cont}
Order the vertices so that each is a circle intersection of two placed neighbours (\emph{tri})
or, when only one neighbour is placed, a \emph{free} angle; the non-tree edges become
closure conditions. This reduces the whole system to the two free angles and the two
closures $g_1,g_2$ of \S\ref{sec:constr}.
\end{algoritmo}

\begin{algoritmo}[Galois classification by Frobenius]\label{alg:gal}
Factor the (validated) minimal polynomial modulo many primes $p\nmid\operatorname{disc}$;
the multiset of factor degrees is the cycle type of the Frobenius at $p$. The collection
of cycle types constrains the Galois group; a $p$-cycle for a suitable prime forces
$A_n$ or $S_n$ by Jordan's theorem \cite{DixonMortimer,Wielandt}, and parity
(even/odd permutations) distinguishes $A_n$ from $S_n$ \cite{Cohen}.
\end{algoritmo}

\begin{definicao}[Frobenius census]\label{def:census}
Let $f\in\Z[x]$ be irreducible of degree $n$. For each prime $p$ of \emph{good
reduction} --- $p\nmid\operatorname{disc}(f)$ and $p$ does not divide the leading
coefficient, so $f\bmod p$ is separable --- factor $f$ over $\mathbb{F}_p$; by
\emph{Dedekind's theorem} \cite{Cohen} the multiset of the degrees of its irreducible
factors equals the cycle type of the Frobenius conjugacy class
$\operatorname{Frob}_p\in\Gal(f)$ acting on the $n$ roots. The \emph{Frobenius census}
is the collection of these cycle types gathered over many primes $p$. By the
\emph{Chebotarev density theorem} \cite{Cohen}, as $p$ ranges over the primes the
relative frequency of each cycle type converges to the proportion of elements of
$\Gal(f)$ having that cycle type; hence the census reconstructs, statistically, the
distribution of conjugacy classes of $\Gal(f)\le S_n$. It yields three levers:
(i) element orders (the $\operatorname{lcm}$ of cycle lengths) reveal the primes
dividing $|\Gal(f)|$ (Cauchy); (ii) a cycle of prime length $q$ with $n/2<q<n-2$
forces $A_n$ or $S_n$ (Jordan \cite{DixonMortimer,Wielandt}); (iii) the presence of an
odd permutation distinguishes $S_n$ from $A_n$. For $EI_{17}$ ($n=20$) the census
exhibits a $17$-cycle (whence $\Gal\supseteq A_{20}$) and a $20$-cycle (odd, whence
$\Gal=S_{20}$).
\end{definicao}

\subsection{Classical results invoked}\label{ssec:classical}
For completeness we state, in academic notation, the classical theorems used above; all
are cited, not proved here.

\begin{teorema}[Dedekind's factorization theorem \cite{Cohen}]\label{teo:dedekind}
Let $f\in\Z[x]$ be monic and irreducible and let $p$ be a prime with
$p\nmid\operatorname{disc}(f)$. If $f\equiv g_1\cdots g_k\pmod p$ with the $g_i\in\mathbb{F}_p[x]$
distinct, monic, irreducible of degrees $d_1,\dots,d_k$, then the Frobenius conjugacy
class $\operatorname{Frob}_p\subseteq\Gal(f)$ acts on the $n=\deg f$ roots as a
permutation of cycle type $(d_1,\dots,d_k)$.
\end{teorema}

\begin{teorema}[Chebotarev density theorem \cite{Cohen}]\label{teo:chebotarev}
Let $L/\Q$ be Galois with group $G$ and let $C\subseteq G$ be a conjugacy class. The set
of primes $p$ unramified in $L$ with $\operatorname{Frob}_p=C$ has natural density
$|C|/|G|$. Consequently the Frobenius census (Definition~\ref{def:census}) reproduces,
in the limit, the class distribution of $\Gal(f)$.
\end{teorema}

\begin{teorema}[Jordan's theorem, restated \cite{DixonMortimer,Wielandt}]\label{teo:jordan-restated}
Let $H\le S_n$ be a primitive permutation group of degree $n$ containing a $q$-cycle for
some prime $q\le n-3$. Then $H\supseteq A_n$; hence $H\in\{A_n,S_n\}$, and $H=S_n$ iff
$H$ contains an odd permutation.
\end{teorema}

\begin{teorema}[Burnside's $p^{a}q^{b}$ theorem \cite{DixonMortimer}]\label{teo:burnside}
Every finite group whose order has the form $p^{a}q^{b}$, with $p,q$ prime, is solvable.
In particular a transitive $G\le S_n$ all of whose element orders are $\{2,3\}$-numbers,
so that $|G|=2^{a}3^{b}$, is solvable.
\end{teorema}

\begin{teorema}[Cardano's formula; \emph{casus irreducibilis} \cite{Cox}]\label{teo:cardano}
Let $t^3+Pt+Q$ be a depressed cubic over a real field, with discriminant quantity
$\Delta=\tfrac{Q^2}{4}+\tfrac{P^3}{27}$. Its real roots are
$t=\sqrt[3]{-\tfrac{Q}{2}+\sqrt{\Delta}}+\sqrt[3]{-\tfrac{Q}{2}-\sqrt{\Delta}}$. If the
cubic is irreducible with three distinct real roots then $\Delta<0$ (the \emph{casus
irreducibilis}): the radicands are non-real, and no expression of the roots by
\emph{real} radicals exists; the real cube root is supplied by the Beloch fold $O6$
(origametry \cite{Hull}), not by ruler and compass.
\end{teorema}

\begin{algoritmo}[Gram--SVD faithfulness certifier]\label{alg:gram}
Given a realization $p$: centre the points, take the SVD of the $n\times2$ coordinate
matrix and check exactly two singular values exceed a tolerance (planar, rank~$2$);
compute $e_{\max}=\max_{uv\in E}\bigl||p(u)-p(v)|-1\bigr|$, the minimal pairwise
separation, and the number of non-edges within tolerance of distance~$1$. Declare
\emph{faithful} iff rank $=2$, $e_{\max}$ below tolerance, separation positive, and no
extra unit distance. This runs in $O(n^2)$ --- the ``easy'' (verification) side of the
$\exists\R$ problem of \S\ref{sec:er}.
\end{algoritmo}

\begin{algoritmo}[Exact referential certificate]\label{alg:refcert}
Order the vertices as in Algorithm~\ref{alg:cont}. For each tri-vertex $v$ from
neighbours $A,B$ emit the two \emph{known} equations \eqref{eq:refcert} (linear
radical axis $+$ unit chord); for each free vertex the single unit-circle condition; for
the base the rational values; for the two closures the unit conditions. The realization
is the unique solution of this triangular chain over $\Q$, so the twenty-nine
construction edges are certified \emph{exactly and referentially}, with small-height
coefficients, independently of the degree-$20$ absolute field
(Proposition~\ref{prop:radaxis}).
\end{algoritmo}

\begin{algoritmo}[Modular census (diagnostic)]\label{alg:modular}
Over $\mathrm{GF}(p)$ propagate the construction with finite-field square roots
(Tonelli--Shanks), sweeping the free angles over the circle points and all sign branches,
and count closures. The count $N(p)$ is a coarse diagnostic only: at ramified/special
primes it is contaminated by degenerate reductions (Remark~\ref{obs:pitfall}), so it does
\emph{not} replace Algorithm~\ref{alg:gal}.
\end{algoritmo}

\subsection{Pseudocode summary}\label{ssec:pseudo}
For reproducibility we record compact pseudocode for the four computational cores;
the remaining algorithms are one-line specialisations of these. We write
$p:\{v_1,\dots,v_{17}\}\to\R^2$ for the realization and $\theta=(\theta_4,\theta_5)$
for the two free angles.

\medskip
\noindent\hrulefill\\
\textsc{Procedure 1 --- High-precision realization} (cf.\ Alg.~\ref{alg:newton},~\ref{alg:gram}).\\
\textit{Input:} six-digit seed $\theta^{(0)}$; target precision $\mathrm{dps}$.\quad
\textit{Output:} faithful realization $p$, or \textsc{reject}.
\begin{enumerate}[topsep=2pt,itemsep=0pt,leftmargin=2.2em,label=\arabic*.]
\item gauge-fix $v_{17}\gets(0,0)$, $v_{11}\gets(-1,0)$; set working precision to $\mathrm{dps}$ digits (\texttt{mpmath}).
\item form the $30$ edge residuals $F(\theta)=\bigl(|p(u)-p(v)|^2-1\bigr)_{uv\in E\setminus\text{base}}$.
\item \textbf{repeat} $\theta\gets\theta-J_F(\theta)^{+}F(\theta)$ (Newton step) \textbf{until} $\|F(\theta)\|<10^{-\mathrm{dps}}$.
\item \textbf{if} \textsc{Gram--SVD}$(p)$ fails (rank $\neq2$, or an extra unit distance) \textbf{return} \textsc{reject}; \textbf{else return} $p$.
\end{enumerate}
\noindent\hrulefill

\medskip
\noindent\hrulefill\\
\textsc{Procedure 2 --- Continuity reduction} (cf.\ Alg.~\ref{alg:cont}).\\
\textit{Input:} adjacency of $EI_{17}$; base edge $(v_{11},v_{17})$.\quad
\textit{Output:} free angles $(\theta_4,\theta_5)$ and closures $g_1,g_2$.
\begin{enumerate}[topsep=2pt,itemsep=0pt,leftmargin=2.2em,label=\arabic*.]
\item mark $v_{11},v_{17}$ \emph{placed}; queue the remaining vertices.
\item \textbf{while} some vertex $v$ is unplaced:
  \begin{enumerate}[topsep=1pt,itemsep=0pt,leftmargin=1.6em,label=(\alph*)]
  \item \textbf{if} $v$ has two placed neighbours $A,B$ (\emph{tri}): place
        $v=\tfrac{A+B}2\pm\sqrt{1-\tfrac{|A-B|^2}4}\,\tfrac{(B-A)^\perp}{|A-B|}$ (radical axis);
  \item \textbf{else if} $v$ has one placed neighbour: introduce a \emph{free} angle $\theta$ for $v$.
  \end{enumerate}
\item each non-tree edge becomes a unit-length \emph{closure} equation.
\item eliminate all but $(\theta_4,\theta_5)$ $\Rightarrow$ closures $g_1(\theta_4,\theta_5),g_2(\theta_4,\theta_5)$; \textbf{return} $(\theta_4,\theta_5;\,g_1,g_2)$.
\end{enumerate}
\noindent\hrulefill

\medskip
\noindent\hrulefill\\
\textsc{Procedure 3 --- Galois classification by Frobenius} (cf.\ Alg.~\ref{alg:gal}).\\
\textit{Input:} validated minimal polynomial $m(x)\in\Z[x]$, $\deg m=n$; prime bound $P$.\quad
\textit{Output:} $\operatorname{Gal}(m)\le S_n$.
\begin{enumerate}[topsep=2pt,itemsep=0pt,leftmargin=2.2em,label=\arabic*.]
\item $C\gets\emptyset$.
\item \textbf{for} each prime $p\le P$ with $p\nmid\operatorname{disc}(m)$: factor $m\bmod p$ and append the multiset of factor degrees (Frobenius cycle type) to $C$.
\item \textbf{if} $C$ contains a $q$-cycle with $q$ prime, $n/2<q<n-2$: then $\operatorname{Gal}\supseteq A_n$ (Jordan).
\item \textbf{if} $C$ contains an odd permutation \textbf{return} $S_n$; \textbf{else return} $A_n$.
\end{enumerate}
\noindent\hrulefill

\medskip
\noindent\hrulefill\\
\textsc{Procedure 4 --- Exact referential certificate} (cf.\ Alg.~\ref{alg:refcert}).\\
\textit{Input:} vertex order from Procedure~2; neighbours $A,B$ of each tri-vertex.\quad
\textit{Output:} triangular system over $\Q$ certifying the $29$ construction edges.
\begin{enumerate}[topsep=2pt,itemsep=0pt,leftmargin=2.2em,label=\arabic*.]
\item $S\gets\{\,v_{11}=(-1,0),\ v_{17}=(0,0)\,\}$ (rational base).
\item \textbf{for} each tri-vertex $v$ with placed $A,B$: add the linear radical axis $2(B-A)\cdot v=|B|^2-|A|^2$ and the unit chord $|v-A|^2=1$.
\item \textbf{for} each free vertex $v$: add the single unit-circle condition; then add the two closure unit-conditions.
\item \textbf{return} $S$ (unique solution over $\Q$; the edges are certified exactly, with small-height coefficients).
\end{enumerate}
\noindent\hrulefill

\begin{observacao}[Two pitfalls, recorded honestly]\label{obs:pitfall}
(i) Low-precision \textsc{PSLQ} returns \emph{spurious} minimal polynomials: degrees $8$
(at $160$ digits), $9$ (at $220$), $10$ (at $330$), \dots, each verifying at the
precision it was found and \emph{collapsing} at higher precision. Only near
$\sim2500$ digits does the true, irreducible degree-$20$ polynomial (leading coefficient
$\sim10^{97}$) appear and persist (Figure~\ref{fig:poly}). (ii) Point counts over
$\mathrm{GF}(p)$ gave $N(p)\in\{0,20,24,\dots\}$, but at $p=23,37$ \emph{every} such
solution is degenerate (coincident vertices mod $p$); the modular count does not measure
the faithful field. Neither shortcut is reliable; the exact irreducible polynomial and
its Frobenius census are.
\end{observacao}

\begin{observacao}[A concrete spurious candidate at machine precision]
\label{obs:spurious10}
The pitfall of (i) can be exhibited explicitly. Applying integer-relation detection
(\texttt{algdep}/PSLQ, or Mathematica's \texttt{RootApproximant}) to the
\emph{machine-precision} coordinate $x_{v_1}\approx-2.49881$ of the clean realization
(the precision at which published embeddings, e.g.\ \textsc{GraphData}, are stored)
yields the monic polynomial
\[
  q(x)=x^{10}+2x^{9}-2x^{8}+4x^{6}-2x^{5}+2x^{4}+6x^{3}-2x^{2}+5,
\]
irreducible over $\Q$, whose real root
$r=-2.498806168621959715\ldots$ agrees with $x_{v_1}$ to about seven digits ---
a thoroughly convincing \emph{false} minimal polynomial. Evaluating the certified
degree-$20$ polynomial $p$ at $r$ gives $p(r)\approx3.8\times10^{92}\neq0$, whereas
$p(x_{v_1})=0$ exactly; the two numbers part ways after the seventh digit. The
information count explains the phenomenon: matching $d$ digits constrains candidates
only up to total height $\sim10^{d}$, so at machine precision a height-$6$, degree-$10$
candidate always exists, while the true polynomial (degree $20$, height $\sim10^{97}$)
requires $\gtrsim20\times97\approx1940$ digits to be seen at all --- hence the
$2600$-digit recomputation, the validation $|p(\alpha)|<10^{-1900}$, and the symbolic
edge check in $\Q(\alpha)$.
\end{observacao}

\section{Results}\label{sec:results}

\subsection{Degree and irreducibility}
Refining the realization (Algorithm~\ref{alg:newton}) to $2600$ digits and applying
Algorithm~\ref{alg:pslq}, the coordinate $x_{v_1}$ satisfies an integer polynomial of
degree $20$, verified to $|p(\alpha)|<10^{-1900}$; the polynomial is \emph{irreducible}
over $\Q$. Hence $[\Q(\alpha):\Q]=20=2^2\cdot5$, and $\Q(\alpha)$ has no proper subfield
(primitive). This equals the degree of the closure resultant, i.e.\ the number of
complex realizations. The coefficient heights reach $\sim10^{97}$
(Figure~\ref{fig:poly}), which is why any integer-relation search needs $\sim2000$
digits and a computer-algebra backend.

\begin{figure}[htbp]
\centering
\includegraphics[width=.8\textwidth]{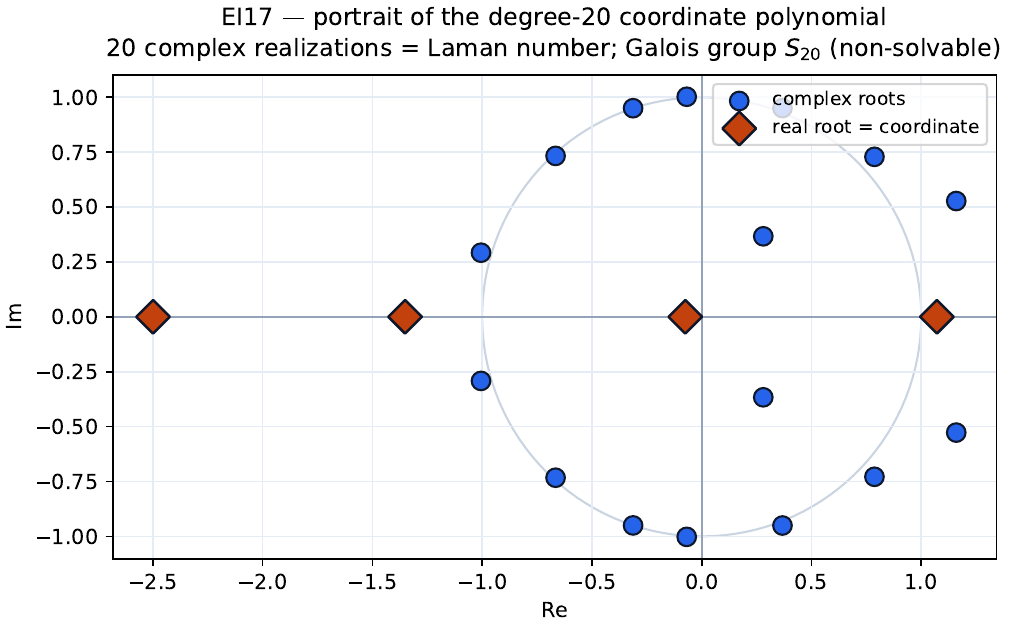}
\caption{Coefficient heights of the degree-$20$ minimal polynomial of a coordinate of
$EI_{17}$: the leading coefficient is $\sim10^{97}$. Such height is the reason
low-precision integer-relation searches fail (Remark~\ref{obs:pitfall}).}
\label{fig:poly}
\end{figure}

\subsection{The Galois group is $S_{20}$}
Factoring the degree-$20$ polynomial modulo $\sim200$ primes (Algorithm~\ref{alg:gal})
gives Frobenius cycle types including a $20$-cycle $(20)$, a $19$-cycle $(1,19)$, a
$17$-cycle $(3,17)$, and every partition of $20$; element orders are divisible by all
primes up to $19$ (Table~\ref{tab:cycles}).

\begin{table}[htbp]
\centering
\caption{A sample of observed Frobenius cycle types (factor degrees mod $p$).}
\label{tab:cycles}
\begin{tabular}{ll}
\toprule
cycle type & meaning \\
\midrule
$(20)$ & $20$-cycle: odd permutation $\Rightarrow$ group $\not\subseteq A_{20}$ \\
$(1,19)$ & $19$-cycle fixing $1$ point \\
$(3,17)$ & $17$-cycle fixing $3$ points $\Rightarrow$ contains $A_{20}$ (Jordan) \\
$(9,11),(6,14),(4,16),\dots$ & orders up to $\mathrm{lcm}$ with primes $\le19$ \\
\bottomrule
\end{tabular}
\end{table}

\begin{lema}\label{lem:galois}
$\Gal(p/\Q)=S_{20}$.
\end{lema}
\begin{proof}
The factorization contains an irreducible factor of degree $17$ --- a $17$-cycle with
three fixed points. As $17$ is prime and $\le n-3=17$, a transitive group of degree $20$
containing such a cycle contains $A_{20}$ (Jordan's theorem \cite{DixonMortimer,Wielandt}).
The factorization also contains a $20$-cycle, which is an \emph{odd} permutation (a cycle
of even length), so the group is not contained in $A_{20}$. Hence $\Gal(p/\Q)=S_{20}$.
\end{proof}

\subsection{Main theorem}
\begin{teorema}\label{teo:principal}
Let $G$ be $EI_{17}$ (\textsc{HoG} 51375) and let $K$ be the coordinate field of the
certified faithful realization of Figure~\ref{fig:real}, with gauge $v_{17}=(0,0)$,
$v_{11}=(-1,0)$. Then a primitive coordinate of $K$ has an irreducible minimal
polynomial of degree $20=2^2\cdot5$ over $\Q$, and the Galois group of its splitting
field is the full symmetric group $S_{20}$ (order $20!$). Consequently:
\begin{enumerate}[leftmargin=2em,itemsep=1pt]
\item $K$ is \emph{not} ruler-and-compass constructible ($20$ is not a power of $2$);
\item since $S_{20}$ is \emph{not solvable} --- its only proper composition factor is the
non-abelian simple group $A_{20}$ --- by Galois's theorem $K$ is \emph{not solvable by
radicals}: it lies in no radical tower $K_{i+1}=K_i(\sqrt[n_i]{a_i})$. In particular $K$
is \emph{not} origami-constructible, at any fold order producing radicals;
\item nevertheless, by Proposition~\ref{prop:radaxis}, every vertex other than
$v_{11},v_{17}$ is an explicit tower of square roots \emph{over} the two free angles
$\theta_4,\theta_5$; the non-radicality is confined to those two angles, the roots of
the degree-$20$ closure resultant.
\end{enumerate}
\end{teorema}

\begin{observacao}[The pair $51375$ / $51376$]\label{obs:pair}
The two triangle-free Exoo--Ismailescu neighbours are opposite extremes of
constructibility. The $19$-vertex graph is origami: degree $12=2^2\cdot3$, Galois group
$12T236$ of order $2304=2^8\cdot3^2$, \emph{solvable}, needing the cubic fold $O6$
\cite{ei19}. The $17$-vertex graph is exotic: degree $20=2^2\cdot5$, Galois group
$S_{20}$, \emph{non-solvable}. In the Laman-number reading (\S\ref{sec:conjecture}),
$\supp(N)=\{2,3\}$ for $51376$ and $\{2,5\}$ for $51375$: the prime $5$ is the exact
signature of the obstruction.
\end{observacao}

\begin{table}[htbp]
\centering
\caption{The constructibility hierarchy read off the coordinate field (Def.~\ref{def:tiers}),
with unit-distance / Laman examples.}
\label{tab:tiers}
\small\setlength{\tabcolsep}{4pt}
\begin{tabular}{lll}
\toprule
tier & Galois group of $K$ & example \\
\midrule
compass, $\supp=\{2\}$ & $2$-group & $\mathrm{UnitDistance}\{21,2\}$: $\Q(\cos\tfrac{2\pi}{7})$, cyclic \\
origami, $\supp=\{2,3\}$ & $\{2,3\}$-group (solvable) & $EI_{19}$: degree $12$, $12T236$ \cite{ei19} \\
exotic, $p\ge5$ & non-$\{2,3\}$ (here non-solvable) & $EI_{17}$: degree $20$, $S_{20}$ (this paper) \\
\bottomrule
\end{tabular}
\end{table}

\subsection{Non-global rigidity and two real realizations}
Being isostatic but not redundantly rigid, $EI_{17}$ is \emph{not} globally rigid ---
Hendrickson's necessary conditions ($3$-connectivity and redundant rigidity) fail
\cite{Connelly,Hendrickson}. Accordingly the degree-$20$ solution set contains
\emph{more than one} non-congruent real faithful realization: besides the rational-based
one certified here, a combinatorially distinct real realization occurs, with a different
pairwise-distance multiset. Since $\Gal(G)=S_{20}$ acts on the twenty complex solutions
as the \emph{full} symmetric group, the realizations are as algebraically independent as
possible --- the antithesis of the small, structured orbit of an origami graph.

\begin{observacao}[A second exotic example: the braced heptagon $\{35,1\}$]\label{obs:hept}
The exotic tier is not peculiar to $EI_{17}$. The braced regular heptagon $\{35,1\}$
(a $19$-vertex, $35$-edge Laman graph) has a closure field ramified at the non-Pierpont
primes $103$ and $587$, its Galois data carrying the factor $2^6\cdot103\cdot587$; it is
likewise beyond compass and single-fold origami. Together with $EI_{17}$ it shows the
exotic category is populated by natural small configurations, and that a prime $\ge5$ in
the realization count is a robust signature of it.
\end{observacao}

\section{The origami neighbour $EI_{19}$ (\textsc{HoG} 51376)}\label{sec:ei19}
We now certify the $19$-vertex neighbour, whose field is the \emph{opposite} tier:
solvable, origami. It has $n=19$, $|E|=35=2n-3$ (isostatic), and is the smallest
state-of-the-art UDG for which an origami obstruction is established exactly \cite{ei19}.
Fourteen \emph{base} vertices $v_1$--$v_{14}$ carry exact coordinates; five \emph{cap}
vertices form a pentagon $v_{15}\cdots v_{19}$ anchored at $v_1,v_2,v_3$
(Figure~\ref{fig:ei19}).

\begin{figure}[t]
\centering
\includegraphics[width=0.60\linewidth]{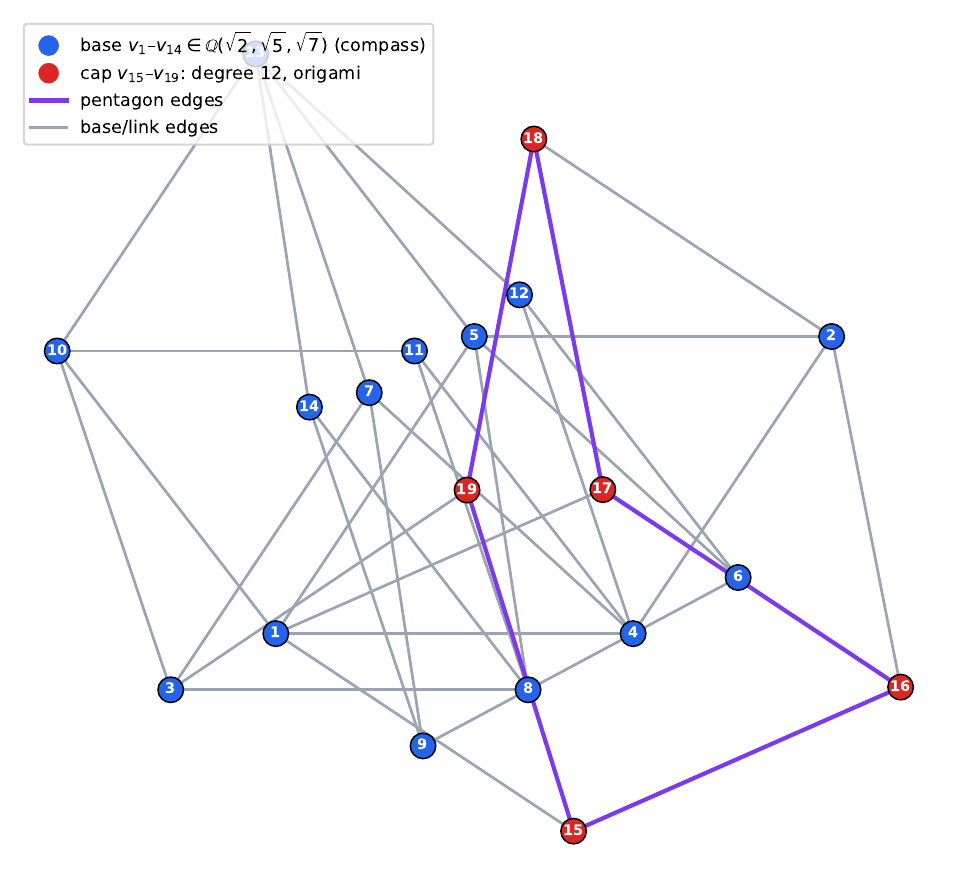}
\caption{Certified faithful realization of $EI_{19}$. Blue: the fourteen base vertices,
exact in $\Q(\sqrt2,\sqrt5,\sqrt7)$ (Table~\ref{tab:ei19coords}). Red: the five cap
vertices, whose field has degree $12$ over $\Q$. Purple edges are internal to the
pentagon.}
\label{fig:ei19}
\end{figure}

\subsection{Plane-geometry certificate of the base}\label{ssec:ei19plane}
As for $EI_{17}$, the construction is purely Euclidean: with the gauge $v_1=(0,0)$,
$v_4=(1,0)$, each base vertex is the intersection of two unit circles centred at placed
neighbours, hence a point on their \emph{radical axis} (Proposition~\ref{prop:radaxis}),
and is fixed \emph{exactly and referentially} by one linear equation
$2(B-A)\cdot v=|B|^2-|A|^2$ and one unit chord $|v-A|^2=1$ (Algorithm~\ref{alg:refcert}, Procedure~4). Solving this
triangular chain over $\Q$ places the
fourteen base coordinates in $\Q(\sqrt2,\sqrt5,\sqrt7)$ (degree $8$, \emph{compass}),
with basis $\{1,\sqrt2,\sqrt5,\sqrt7,\sqrt{10},\sqrt{14},\sqrt{35},\sqrt{70}\}$; the $24$
edges internal to the base then give squared distance $1$ \emph{identically} by exact
arithmetic (Table~\ref{tab:ei19coords}). Thus the whole base is certified by plane
geometry alone, independently of the degree-$12$ cap field --- exactly the optional
referential route of \S\ref{sec:methods}.

\begin{table}[t]
\centering\small\renewcommand{\arraystretch}{1.3}
\caption{Exact coordinates of the fourteen $EI_{19}$ base vertices ($v_1=(0,0)$,
$v_4=(1,0)$), all in $\Q(\sqrt2,\sqrt5,\sqrt7)$; the $24$ internal edges verified
$=1$ by exact arithmetic.}
\label{tab:ei19coords}
\begin{tabular}{@{}clll@{}}
\toprule
$v$ & $x_v$ & $y_v$ & Field \\ \midrule
$v_{1}$ & $0$ & $0$ & $\Q$\\
$v_{2}$ & $\tfrac{14}{9}$ & $\tfrac{2}{9}\sqrt{14}$ & $\Q(\sqrt{14})$\\
$v_{3}$ & $-\tfrac{1}{9}\sqrt7$ & $-\tfrac{1}{9}\sqrt2$ & $\Q(\sqrt2,\sqrt7)$\\
$v_{4}$ & $1$ & $0$ & $\Q$\\
$v_{5}$ & $\tfrac{5}{9}$ & $\tfrac{2}{9}\sqrt{14}$ & $\Q(\sqrt{14})$\\
$v_{6}$ & $1+\tfrac{1}{9}\sqrt7$ & $\tfrac{1}{9}\sqrt2$ & $\Q(\sqrt2,\sqrt7)$\\
$v_{7}$ & $\tfrac{5}{9}-\tfrac{1}{9}\sqrt7$ & $-\tfrac{1}{9}\sqrt2+\tfrac{2}{9}\sqrt{14}$ & $\Q(\sqrt2,\sqrt7)$\\
$v_{8}$ & $1-\tfrac{1}{9}\sqrt7$ & $-\tfrac{1}{9}\sqrt2$ & $\Q(\sqrt2,\sqrt7)$\\
$v_{9}$ & $1-\tfrac{2}{9}\sqrt7$ & $-\tfrac{2}{9}\sqrt2$ & $\Q(\sqrt2,\sqrt7)$\\
$v_{10}$ & $-\tfrac{1}{18}\sqrt7-\tfrac{1}{18}\sqrt{70}$ & $-\tfrac{1}{18}\sqrt2+\tfrac{7}{18}\sqrt5$ & $\Q(\sqrt2,\sqrt5,\sqrt7)$\\
$v_{11}$ & $1-\tfrac{1}{18}\sqrt7-\tfrac{1}{18}\sqrt{70}$ & $-\tfrac{1}{18}\sqrt2+\tfrac{7}{18}\sqrt5$ & $\Q(\sqrt2,\sqrt5,\sqrt7)$\\
$v_{12}$ & $1+\tfrac{1}{18}\sqrt7-\tfrac{1}{18}\sqrt{70}$ & $\tfrac{1}{18}\sqrt2+\tfrac{7}{18}\sqrt5$ & $\Q(\sqrt2,\sqrt5,\sqrt7)$\\
$v_{13}$ & $\tfrac{5}{9}-\tfrac{1}{18}\sqrt7-\tfrac{1}{18}\sqrt{70}$ & $-\tfrac{1}{18}\sqrt2+\tfrac{7}{18}\sqrt5+\tfrac{2}{9}\sqrt{14}$ & $\Q(\sqrt2,\sqrt5,\sqrt7)$\\
$v_{14}$ & $1-\tfrac{1}{6}\sqrt7-\tfrac{1}{18}\sqrt{70}$ & $-\tfrac{1}{6}\sqrt2+\tfrac{7}{18}\sqrt5$ & $\Q(\sqrt2,\sqrt5,\sqrt7)$\\
\bottomrule
\end{tabular}
\end{table}

\subsection{The cap, the single free angle, and the origami theorem}
The five cap vertices are a function of a \emph{single} angle $\theta$ (the position of
$v_{15}$ on the unit circle about $v_1$): writing the unit-circle intersection (fold) as
$\Phi_\varepsilon(A,B)=\tfrac{A+B}{2}+\varepsilon\sqrt{1-\tfrac14\lVert A-B\rVert^2}\,
\tfrac{R(B-A)}{\lVert B-A\rVert}$, $R(x,y)=(-y,x)$,
\[
v_{15}=(\cos\theta,\sin\theta),\ v_{16}=\Phi_{+}(v_2,v_{15}),\ v_{17}=\Phi_{+}(v_1,v_{16}),\
v_{18}=\Phi_{-}(v_2,v_{17}),\ v_{19}=\Phi_{+}(v_3,v_{15}),
\]
and the realization closes through the single edge $(18,19)$:
$g(\theta)=\lVert v_{18}-v_{19}\rVert^2-1=0$, with $\theta=-33.5528\ldots^\circ$ and
orientations $(+,+,-,+)$ (Figure~\ref{fig:ei19constr}). Here \emph{one} free angle and
\emph{one} closure (against \emph{two} of each for $EI_{17}$) is the structural reason
the neighbour stays solvable.

\begin{figure}[t]
\centering
\includegraphics[width=\linewidth]{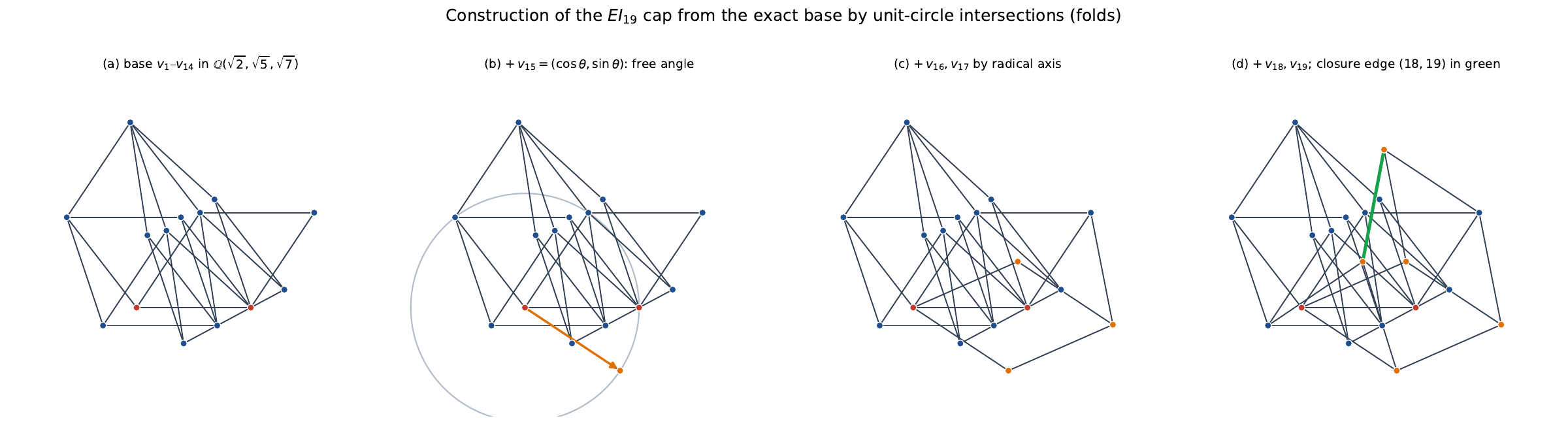}
\caption{Construction of the $EI_{19}$ cap from the exact base by unit-circle
intersections (folds): (a) base; (b) free vertex $v_{15}$ on the circle about $v_1$;
(c) $v_{16},v_{17}$; (d) $v_{18},v_{19}$ and the single closure edge $(18,19)$ (green),
which determines $\theta$.}
\label{fig:ei19constr}
\end{figure}

\begin{teorema}[Origami obstruction for $EI_{19}$ \cite{ei19}]\label{teo:ei19}
The coordinate $x_{15}=\cos\theta$ has minimal polynomial over $\Q$
\begin{multline*}
234365481\,x^{12}-1458274104\,x^{11}+2517258870\,x^{10}+2679075000\,x^{9}
-11544205617\,x^{8}\\
-81856656\,x^{7}+39529761684\,x^{6}-66881540208\,x^{5}+52962075495\,x^{4}\\
-22124598328\,x^{3}+4459486710\,x^{2}-292482120\,x+950177,
\end{multline*}
\emph{irreducible} of degree $12=2^2\cdot3$. Hence $[\Q(x_{15}):\Q]=12$ is not a power of
$2$ (not compass), and the Galois group is the \emph{solvable} transitive group
$G=12T236$, $|G|=2304=2^8\cdot3^2$, with subfield tower
$\Q\subset\Q(\sqrt7)\subset K_6\subset\Q(x_{15})$ of degrees $2,3,2$ and
$K_6=\Q[t]/(t^6-2t^5-101t^4+62t^3+1514t^2-3616t+2304)$. Therefore $x_{15}$ is
\emph{origami-constructible} and the cubic Beloch fold $O6$ is \emph{necessary}
(origametry \cite{Hull}).
\end{teorema}

A Frobenius census over $2255$ good primes (Table~\ref{tab:ei19census}) shows every cycle
length is $2^i3^j$ and every element order lies in $\{1,2,3,4,6,8,12\}$; by Cauchy the
primes dividing $|G|$ are $\{2,3\}$ and by Burnside $p^aq^b$ the group is solvable ---
$G=12T236$ identified exactly in SageMath/PARI.

\begin{table}[t]
\centering\small\renewcommand{\arraystretch}{1.12}
\caption{Frobenius cycle-type census of the $EI_{19}$ polynomial over $2255$ good primes.
Every cycle length is $2^i3^j$; orders in $\{1,2,3,4,6,8,12\}$; group $12T236$,
$|G|=2304$.}
\label{tab:ei19census}
\begin{tabular}{@{}lll@{}}
\toprule
cycle type & order & density \\ \midrule
$(2,2,4,4)$ & $4$ & $0.219$\\
$(6,6)$ & $6$ & $0.188$\\
$(4,8)$ & $8$ & $0.131$\\
$(1,1,2,2,2,4)$ & $4$ & $0.074$\\
$(1,1,2,2,3,3)$ & $6$ & $0.061$\\
$(2,2,2,6)$ & $6$ & $0.057$\\
$(1^4,2^4)$ & $2$ & $0.050$\\
$(2^6)$ & $2$ & $0.049$\\
$(2,3,3,4)$ & $12$ & $0.048$\\
$(1,1,4,6)$ & $12$ & $0.044$\\
$(3,3,3,3)$ & $3$ & $0.025$\\
$(1^{12})$ & $1$ & $0.0013$\\
\bottomrule
\end{tabular}
\end{table}

\begin{observacao}[Radical tower and \emph{casus irreducibilis}]\label{obs:ei19casus}
The $2$--$3$--$2$ tower gives the radical form of $x_{15}$ ($\approx0.833377$). The cubic
step $\Q(\sqrt7)\subset K_6$ is generated by the real root $\beta$ of
$\beta^3-(1+2\sqrt7)\beta^2-(37+5\sqrt7)\beta+(64+16\sqrt7)=0$, whose Cardano discriminant
$\Delta=\tfrac{Q_0^2}{4}+\tfrac{P_0^3}{27}=-\tfrac{14126}{3}-\tfrac{10577\sqrt7}{6}<0$ puts
it in \emph{casus irreducibilis}: all three roots real, yet the formula needs cube roots
of complex numbers. Hence $x_{15}$, though real, admits \emph{no real radicals}; the cube
root is unavoidable --- the trisection obstruction, resolved by $O6$ and not by ruler and
compass. Finally $x_{15}=\tfrac{-P+\sqrt{P^2-4Q}}{2}$ over $K_6=\Q(\sqrt7,\beta)$ with $P,Q$
of small height in $\beta$. This is the exact opposite of $EI_{17}$, where no radical tower
exists at all ($S_{20}$).
\end{observacao}

\begin{observacao}[The difficulty is \emph{planar}, for both]\label{obs:ei19planar}
In $\R^3$, $EI_{19}$ (like $EI_{17}$) is far from isostatic: generic rigidity needs
$3n-6=51$ edges but there are $35$, giving $16$ internal degrees of freedom; the
realization variety is positive-dimensional and the field collapses. The degree-$12$
(resp.\ degree-$20$) difficulty is a phenomenon of the plane.
\end{observacao}

\begin{observacao}[Algebraic vs.\ constructive certificates; the \textsc{GraphData}
asymmetry]\label{obs:graphdata}
Mathematica's \textsc{GraphData} stores an exact embedding of $EI_{19}$: fourteen
vertices in closed radicals over $\Q(\sqrt2,\sqrt5,\sqrt7)$ and five as \texttt{Root}
objects of degrees $6$, $12$ and $24$ --- all $\{2,3\}$-numbers (Pierpont), as the
solvability of $12T236$ predicts. This is, however, only the \emph{algebraic}
certificate: a \texttt{Root} object names a number without exhibiting a construction.
Our construction supplies the strictly stronger \emph{constructive} certificate, in
Hull's origametry \cite{Hull}: an explicit Huzita--Hatori fold sequence realizing the
$\{2,3\}$-tower --- $O5$ folds for the quadratic layers and a single cubic Beloch fold
$O6$ --- so the tier (origami, not merely ``radical'') is established geometrically.
The asymmetry of the database is itself an empirical illustration of the dichotomy: for
$EI_{17}$ no radical (hence no \texttt{Root}-object of tolerable height) exists at all
--- its minimal polynomial has height $\sim10^{97}$ and Galois group $S_{20}$ --- and
\textsc{GraphData} accordingly stores that embedding only in floating point, which is
precisely the regime where spurious candidates such as
Remark~\ref{obs:spurious10} arise.
\end{observacao}

\section{The Laman-number conjecture}\label{sec:conjecture}
Over the generic base $K_0$ the field of one realization has degree exactly the Laman
number $N(G)$ \cite{Capco,GKKPS}, and one always has
$N\mid|\Gal(G)|\mid N!$, so $\supp(N)\subseteq\supp(|\Gal(G)|)$. This inclusion already
\emph{detects} obstructions: a prime $\ge5$ dividing $N$ forces a non-$\{2,3\}$ factor of
$|\Gal(G)|$. We conjecture the sharp form.

\begin{observacao}[Laman number as a constructibility sieve, conjecturally]\label{obs:conj}
For a Laman graph with \emph{non-hidden closures} (each Henneberg-2 closure of prime
degree with cyclic Galois group), $\supp(N)=\supp(|\Gal(G)|)$, and the tier of $K$ is
read from the prime support of $N=\prod_p p^{e_p}$ as in Definition~\ref{def:tiers}.
\end{observacao}

The two neighbours instantiate this from both sides. For $51376$,
$N=12=2^2\cdot3$ and $|\Gal|=2^8\cdot3^2$: the equality $\supp=\{2,3\}$ \emph{holds}, and
the sieve reads origami \cite{ei19}. For $51375$, $N=20=2^2\cdot5$: the sieve reads
exotic ($5\mid N$), correctly. Here, however, $|\Gal|=|S_{20}|=20!$ has
$\supp=\{2,3,5,7,11,13,17,19\}\supsetneq\{2,5\}$: the \emph{equality} fails --- $S_{20}$
\emph{hides} the primes $3,7,11,13,17,19$. Thus $51375$ is the \emph{extreme}
hidden-closure case, opposite to a quartic-$S_4$ counterexample (which hides a single
$3$); it violates the equality maximally, yet the tier is still detected because the
obstructing prime $5$ is already visible in $N$. The pair both confirms the detection
direction and delimits the domain of the conjecture.

\section{Complexity: an $\exists\R$-completeness witness}\label{sec:er}
Deciding whether a graph has a faithful UDG (or an isostatic) realization is
$\exists\R$-complete \cite{Schaefer,AbelDemaine}. The relevant chain of classes is
\[
  \mathrm{P}\ \subseteq\ \mathrm{NP}\ \subseteq\ \exists\R\ \subseteq\ \mathrm{PSPACE},
\]
where $\exists\R$ is the class of problems polynomial-time reducible to deciding the
truth of an existential first-order sentence over the reals. Integer- or
bounded-coordinate variants are already NP-hard \cite{Saxe}; the general realizability
sits at the $\exists\R$ level, a home for the \emph{universality} phenomena of Mnëv and
of Kapovich--Millson (arbitrary semialgebraic sets arise as realization spaces of
linkages) \cite{Mnev,KapovichMillson}, and realizations may require coordinates of
doubly-exponential bit-size \cite{GPS}. The smallest triangle-free $4$-chromatic UDG
gives a concrete, minimal witness: while \emph{verifying} a given realization is
$O(n^2)$ (the Gram--SVD certifier of Algorithm~\ref{alg:newton}), \emph{finding} one is
hard --- in a $9000$-start complex-Newton experiment on the $30$-variable system, the
degenerate configurations dominate and the faithful realization appears as an isolated
point (found $\sim$ once). The exotic $S_{20}$ field and the $\exists\R$ hardness are two
faces of the same isolation.

\section{Conclusion}
The smallest triangle-free $4$-chromatic unit-distance graph, far from being tame, has
the \emph{most generic} possible coordinate field: a degree-$20$ extension with Galois
group $S_{20}$, not solvable by radicals. Its faithful realization is nonetheless
Euclidean \emph{relative} to two free angles, via the transparent radical-axis
construction of \S\ref{sec:constr}; the whole obstruction is concentrated in those two
angles. Together with its origami neighbour $51376$ \cite{ei19}, the graph $51375$ shows
that the family of small extremal unit-distance graphs already realizes both ends of the
compass/origami/exotic hierarchy --- and, via \S\ref{sec:er}, that this arithmetic
hardness is the shadow of an $\exists\R$-complete search. Figure~\ref{fig:roadmap}
summarizes the whole path, from the object to its Galois group.

\begin{figure}[htbp]
\centering
\includegraphics[width=\textwidth]{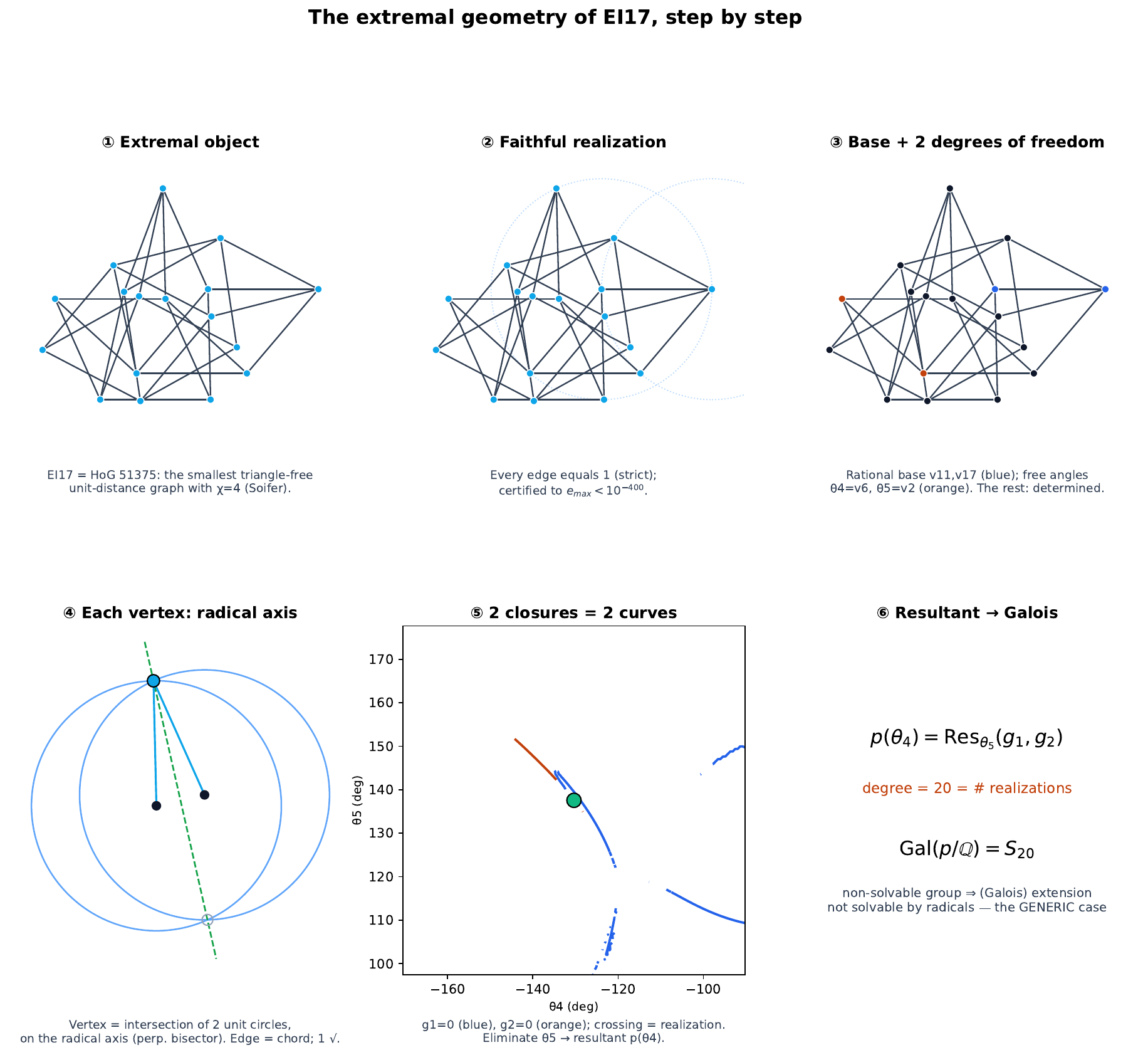}
\caption{The extremal geometry of $EI_{17}$, from the object to the Galois group:
\textcircled{1} the smallest triangle-free $\chi=4$ UDG; \textcircled{2} its faithful
realization; \textcircled{3} rational base and two free angles; \textcircled{4} each
vertex at the intersection of two unit circles (radical axis, Prop.~\ref{prop:radaxis});
\textcircled{5} the two closures as curves whose crossing is the realization;
\textcircled{6} the degree-$20$ resultant with Galois group $S_{20}$, non-solvable.}
\label{fig:roadmap}
\end{figure}

\clearpage
\section*{Acknowledgments}
The author thanks Prof.\ Dr.\ \textbf{Am\'erico Cunha da Silva~Jr.}\ (UERJ) for his
guidance on the origametry part of this work.

It is also a pleasure to thank \textbf{Edward Pegg~Jr.}\ (Wolfram Research), whose questions on
the coordinate fields --- and their degrees --- of faithful realizations of small Laman
and unit-distance graphs are the origin of this entire investigation. His encyclopedic
knowledge of the smallest extremal configurations, the precise problems he pointed the
author to, and his generous and patient correspondence shaped this work at every stage;
the one-angle, radical-axis viewpoint adopted here grew directly out of that exchange. The
author is deeply grateful for his encouragement and guidance. Remaining errors are the
author's own.

\appendix
\section{Coordinates of the certified realization}\label{ap:coords}
Table~\ref{tab:coords} lists \emph{our} certified coordinates of the faithful
realization (Figure~\ref{fig:real}) to twenty-four significant digits, in the gauge
$v_{17}=(0,0)$, $v_{11}=(-1,0)$. A faithful realization is produced \emph{ab initio} by
the continuity reduction of \S\ref{sec:constr} (Algorithm~\ref{alg:cont}: a systematic
sweep of sign branches with closure solving, from generic seeds --- \emph{no} external
input), and refined by Newton (Algorithm~\ref{alg:newton}) to $>2500$ digits. An
approximate embedding of this graph also appears on MathWorld / \textsc{House of
Graphs}~51375 (E.~Pegg~Jr.); our certified coordinates agree with it but \emph{do not
depend on it}. Since the coordinate field is a graph invariant (all faithful
realizations are $S_{20}$-conjugate, \S\ref{sec:results}), the degree-$20$ / $S_{20}$
result is independent of this choice; the strict Gram--SVD certificate (all $31$ edges
$=1$ to $<10^{-400}$, no non-edge at distance~$1$) is original to this work. Each
coordinate is algebraic of degree $20$ over $\Q$, a root of the polynomial of
Appendix~\ref{ap:poly} selected by these digits. Note the four horizontal unit edges
$v_{11}v_{17}$, $v_2v_{10}$, $v_3v_{12}$, $v_6v_{16}$ (difference $(1,0)$), the signature
of the unit rhombi (\S\ref{sec:constr}).

\begin{table}[htbp]
\centering\small
\caption{Our certified coordinates of the faithful realization ($24$ significant digits).}
\label{tab:coords}
\begin{tabular}{rll}
\toprule
$v$ & $x$ & $y$ \\
\midrule
$1$ & $-2.498806168621063073781172$ & $-0.549687325764897172971352$ \\
$2$ & $-2.385349777286040586649497$ & $-0.087042289116627046343162$ \\
$3$ & $-1.976688116468964503756166$ & $-0.999728241116672384265386$ \\
$4$ & $-1.855664803158276574864903$ & $\phantom{-}0.216060141766444058530457$ \\
$5$ & $-1.760489090853103061928075$ & $-0.023378927949119585938183$ \\
$6$ & $-1.647529953519117279782757$ & $-0.762039998487959826837054$ \\
$7$ & $-1.624860686432937524721422$ & $-0.063663361167507460404979$ \\
$8$ & $-1.612262559956085560912846$ & $-1.012332362413167299599543$ \\
$9$ & $-1.408661660817076082893331$ & $\phantom{-}0.912685952000045337922224$ \\
$10$ & $-1.385349777286040586649497$ & $-0.087042289116627046343162$ \\
$11$ & $-1$ & $\phantom{-}0$ \\
$12$ & $-0.976688116468964503756166$ & $-0.999728241116672384265386$ \\
$13$ & $-0.969121194493299061996577$ & $-0.246584894881826068097732$ \\
$14$ & $-0.886543608664977512868325$ & $\phantom{-}0.462645036648270126628190$ \\
$15$ & $-0.738317077767960011853096$ & $-0.526308397815777587033169$ \\
$16$ & $-0.647529953519117279782757$ & $-0.762039998487959826837054$ \\
$17$ & $\phantom{-}0$ & $\phantom{-}0$ \\
\bottomrule
\end{tabular}
\end{table}

\input{appendix_poly}

\input{appendix_construction}

\end{document}

%% file: appendix_poly.tex
\section{The degree-\texorpdfstring{$20$}{20} minimal polynomial}\label{ap:poly}
A primitive coordinate ($x_{v_1}$, gauge of Appendix~\ref{ap:coords}) is a root of the
irreducible, primitive integer polynomial $p(x)=\sum_{k=0}^{20}c_k\,x^k$ with the
coefficients below (leading $c_{20}\sim10^{97}$; found by \texttt{algdep} at $2600$ digits
and verified $|p(\alpha)|<10^{-1900}$). Its Galois group is $S_{20}$ (Lemma~\ref{lem:galois}).

{\footnotesize\begin{flushleft}
$c_{20}={}$\,\seqsplit{10846277150822505740639962570558213942763119983043461674845448343918212271056651745480014057747467}\par\smallskip
$c_{19}={}$\,\seqsplit{19180475212539477649458780252280766370785104165204168446975900051667169039146647895248310235378771}\par\smallskip
$c_{18}=-{}$\,\seqsplit{26070936888654617384259635465342351757876050482308610020563950635118095256174246419408065866234597}\par\smallskip
$c_{17}=-{}$\,\seqsplit{3962009254586100531018427473429740363965658407149928691646435897841056035927607923182216019350940}\par\smallskip
$c_{16}={}$\,\seqsplit{18166859496891545508566917411230114971902727740304713885900161305285905262350527321353616634759110}\par\smallskip
$c_{15}=-{}$\,\seqsplit{7607431700536231966080908557596836755687226593135326666832499865701187554080177241394973984057350}\par\smallskip
$c_{14}={}$\,\seqsplit{45351587309956624267644739061100632322504856432748648077436547828253855443955203136258124315668464}\par\smallskip
$c_{13}=-{}$\,\seqsplit{9408123675708568099217732631464064702259192795720169395209699138778275488455547776691534292051591}\par\smallskip
$c_{12}={}$\,\seqsplit{14682673355886588168945506517818998475845345062830936736552073931495949256607496993509908842889085}\par\smallskip
$c_{11}={}$\,\seqsplit{15143426019910260147689787064208205709627636321450044137594530613944920637911243716728461040656088}\par\smallskip
$c_{10}={}$\,\seqsplit{65968618518965545764431180875089192219762433065476009078893452297224262471375559826743563580939}\par\smallskip
$c_{9}={}$\,\seqsplit{6933435324110256159086227291350057737956789050172029740276872060533462997921572243016990516167855}\par\smallskip
$c_{8}=-{}$\,\seqsplit{9790517400576717511906222722622431562355763122979780641991852086843847078118813130413048942134112}\par\smallskip
$c_{7}={}$\,\seqsplit{6244397981079102575268653077249642024975028707259943515274717143990817911425984104440579854365874}\par\smallskip
$c_{6}=-{}$\,\seqsplit{23586454334648278938656158865770452705897027029874210958571764632770769116482786067681998817533516}\par\smallskip
$c_{5}=-{}$\,\seqsplit{53392439028840081073638800831628946914310140652692035611147137997008514386412828676786323411211286}\par\smallskip
$c_{4}=-{}$\,\seqsplit{26449839798917337714527054796537315591594669043182637188470470315247453343648348898405567316167209}\par\smallskip
$c_{3}=-{}$\,\seqsplit{67496405170016508230749607293466456667836751966863744185909998225730057226496683094528284790742964}\par\smallskip
$c_{2}={}$\,\seqsplit{29801792463858164957001377196524717774830567182899091859788291738590738838397086829441065048179864}\par\smallskip
$c_{1}={}$\,\seqsplit{0}\par\smallskip
$c_{0}=-{}$\,\seqsplit{1312711393788876800678365780809286777142912632161468883999549147019329458866534946330375157169324}\par\smallskip
\end{flushleft}}

%% file: appendix_construction.tex
\section{The geometric construction of \texorpdfstring{$EI_{19}$}{EI19},
from base to cap: a companion walkthrough}\label{ap:construction}

This appendix is a self-contained, illustrated walkthrough of the geometric
construction certified in \S\ref{sec:ei19}: rational gauge, unit rhombus,
quilt of rhombi and the whole base in $\Q(\sqrt2,\sqrt5,\sqrt7)$; then the
cap, the free angle and the closure edge; the full degree-2 system with the
radical axes as lines; where the cubic is born in the elimination; Cardano's
formula, the \emph{casus irreducibilis} and the field $K_6$; the
Huzita--Hatori axioms, Lill's method and trisection; the Henneberg
decomposition of every step; and the orientation signs behind the difficulty
of the certificate. \subsection{The problem}
A \emph{unit-distance graph} (UDG) in the plane is a graph $G=(V,E)$ with
$p:V\to\R^2$ such that $\lVert p(u)-p(v)\rVert=1$ exactly when $uv\in E$.
The graph $EI_{19}$ has $19$ vertices and $35=2\cdot19-3$ edges: it is
\emph{isostatic} (minimally rigid). Once a rational frame is fixed, its
vertices are an \emph{isolated solution} of a polynomial system with rational
coefficients; hence their coordinates are \emph{algebraic} (Maehara's
theorem). The question: \emph{in which number field do they live?} The answer,
for the cap, is a field of degree $12$, \emph{origami}-constructible.

\subsubsection{Maehara's theorem: rigid distances are algebraic}
Why are the coordinates \emph{algebraic}? By \textbf{Maehara's theorem}: the
set $\Gamma$ of distances that occur between two vertices of some \emph{rigid}
UDG in the plane is exactly the set $A^{+}$ of positive algebraic numbers,
$\Gamma=A^{+}$. The direction we use is $\Gamma\subseteq A^{+}$: once a
rational gauge is fixed, a rigid graph has its vertices as an \emph{isolated
solution} of a polynomial system with rational coefficients; hence its
coordinates --- and all mutual distances --- are \emph{algebraic}. This is
what makes ``the coordinate field'' $K$ a genuine \emph{number field}. The
converse $A^{+}\subseteq\Gamma$ is constructive (Maehara realizes any
algebraic number by assembling Kempe linkages). The present work solves the
\emph{exact} converse for one fixed graph: \emph{which} algebraic numbers are
the coordinates of $EI_{19}$ --- and the answer is the degree-$12$ field,
origami.

\subsubsection{Triangle-free}
$EI_{19}$ is \textbf{triangle-free} (girth $4$): no three vertices are
pairwise at distance $1$. This places it at the frontier of the
\emph{Hadwiger--Nelson} problem (the chromatic number of the plane,
$\chi(\R^2)\in\{5,6,7\}$): the Exoo--Ismailescu family of \emph{triangle-free}
UDGs with chromatic number $4$ is what drives that problem's lower bounds.
Being triangle-free changes the arithmetic of the construction: the rigid
pieces are not equilateral triangles (which would give $\sqrt3$ at once) but
\emph{rhombi} ($4$-cycles); so the construction proceeds by \emph{radical
axes} and \emph{free angles} --- leaving the field ``open'' until the closure,
which is exactly where the high-degree arithmetic enters.

\subsubsection{Complexity: NP, ETR and PSPACE --- and the certificate}
Deciding whether a graph admits a faithful (or isostatic) UDG realization is
\emph{$\exists\R$-complete} (Schaefer). The chain of classes is
\[
\mathrm{P}\ \subseteq\ \mathrm{NP}\ \subseteq\ \exists\R\ \subseteq\ \mathrm{PSPACE},
\]
where $\exists\R$ (the level of the \emph{existential theory of the reals},
ETR) is the class of problems that are polynomial-time reducible to deciding the truth
of an existential first-order sentence over $\R$; by Tarski--Seidenberg and
Canny, $\mathrm{ETR}\in\mathrm{PSPACE}$. Realizations may require coordinates
of \emph{doubly-exponential} bit-size. In this setting the present work is the
\textbf{certificate}: \emph{verifying} a given realization is easy ($O(n^2)$,
the Gram--SVD certifier), but \emph{finding} and \emph{exactly certifying} the
coordinate field is the hard side --- and that is what we do, with the exact
minimal polynomial and the Frobenius census. The geometric construction below
is the ``visible'' face of that certificate.

\subsection{The base: gauge, rhombus and quilt}

\subsubsection{Gauge and the rhombus}
Fix the gauge $v_1=(0,0)$, $v_4=(1,0)$. The first two vertices, $v_5$ and
$v_2$, form with $v_1,v_4$ a \textbf{unit rhombus} (a side-$1$
parallelogram): once $v_5=(\cos\theta,\sin\theta)$ is chosen on the circle
about $v_1$, the vertex $v_2$ follows \emph{for free} by translation,
$v_2=v_5+(v_4-v_1)$ (Figure~\ref{fig:losango}).

\begin{figure}[htbp]
\centering
\includegraphics[width=.86\textwidth]{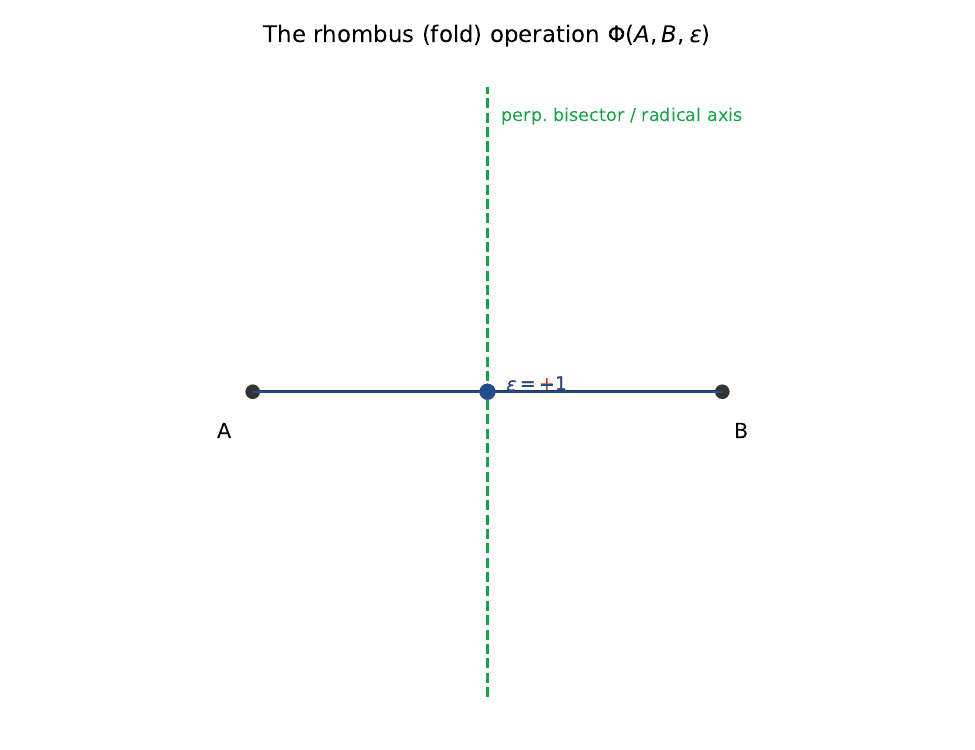}
\caption{The rhombus as a vector operation: the parallelogram's diagonals
bisect each other ($v_1+v_2=v_5+v_4$), hence $v_2=v_5+v_4-v_1$. This is the law
$4\cdot1^2=d_1^2+d_2^2$ (squares of the sides $=$ squares of the diagonals).}
\label{fig:losango}
\end{figure}

\subsubsection{The base as a quilt of 9 rhombi}
The whole base ($14$ vertices) is an \textbf{interlocked quilt of $9$ unit
rhombi} (Figure~\ref{fig:colcha}). Each rhombus satisfies $d_1^2+d_2^2=4$;
where three of its vertices are already placed, the fourth follows by vector
addition. Where no rhombus is available, the vertex follows by \emph{radical
axis} (intersection of two unit circles).

\begin{figure}[htbp]
\centering
\includegraphics[width=.66\textwidth]{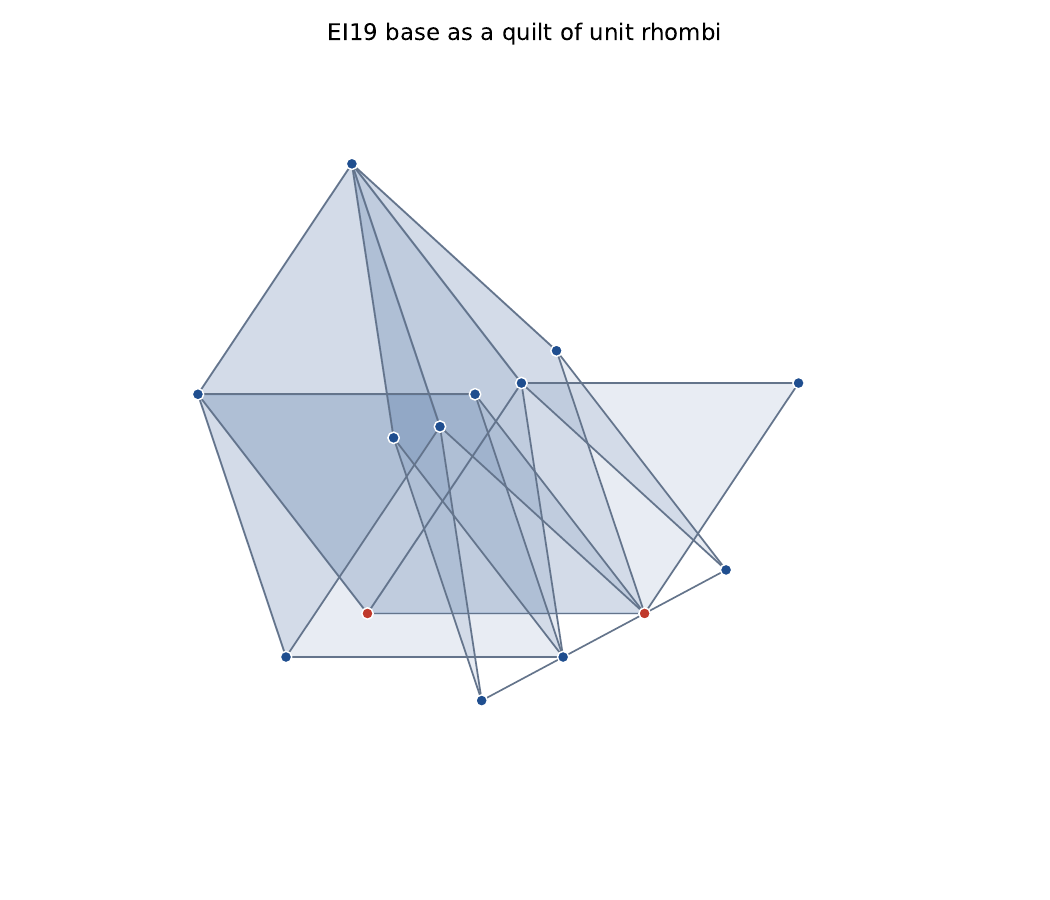}
\caption{The $9$ rhombi $L1$--$L9$ of the base. They share the ``spine''
$v_{10},v_{13}$; this is why the base needs \emph{two} free angles, locked by
closures.}
\label{fig:colcha}
\end{figure}

\subsubsection{Vector construction and exact coordinates}
The base is built in $14$ steps (Figure~\ref{fig:passos}), reproducing the
exact coordinates to machine precision ($\sim10^{-16}$). They all lie in
$K_0:=\Q(\sqrt2,\sqrt5,\sqrt7)$ (degree $8$ over $\Q$), the
\emph{ruler-and-compass} tier (Figure~\ref{fig:angcoord}).

\begin{figure}[htbp]
\centering
\includegraphics[width=\textwidth]{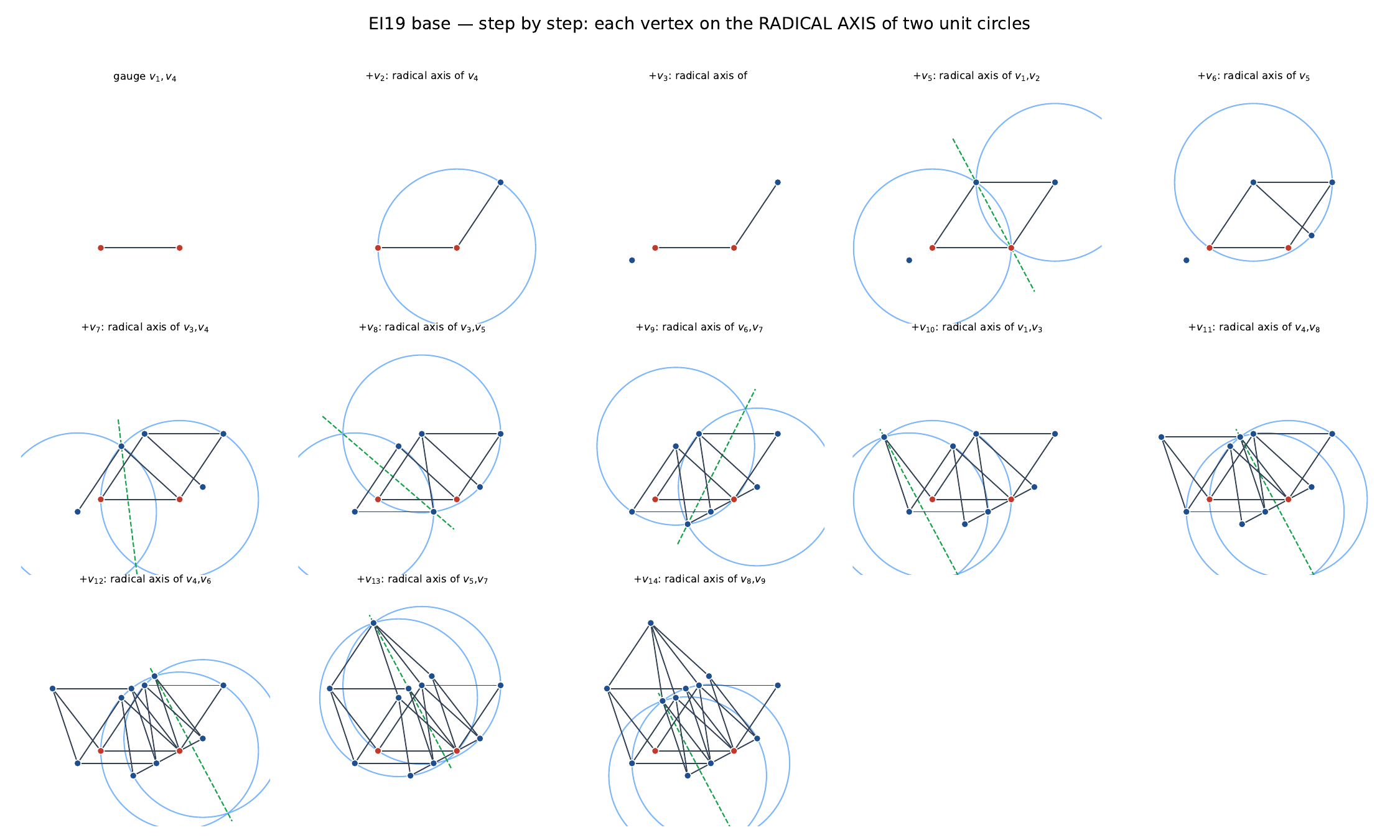}
\caption{Step-by-step vector construction of the base ($12$ panels): green
circle $=$ free angle; orange rhombus $=$ vector sum; two circles $+$ line $=$
radical axis.}
\label{fig:passos}
\end{figure}

\begin{figure}[htbp]
\centering
\includegraphics[width=.92\textwidth]{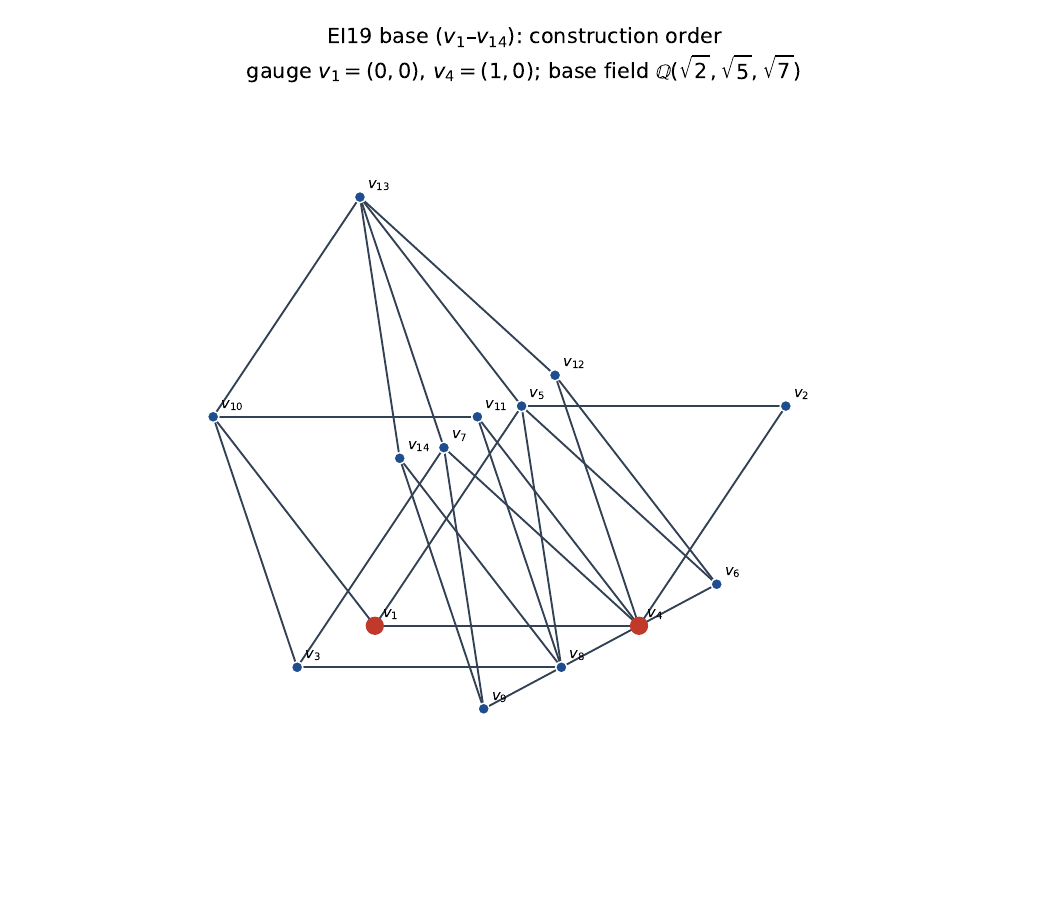}
\caption{Construction-order map of the base: $2$ gauge ($v_1,v_4$) $+$ $2$
free angles ($v_5,v_6$) $+$ $10$ vertices by radical axis. Arrows go from the
two ``parents'' (circle centres) to the child.}
\label{fig:ordem}
\end{figure}

\begin{figure}[htbp]
\centering
\includegraphics[width=\textwidth]{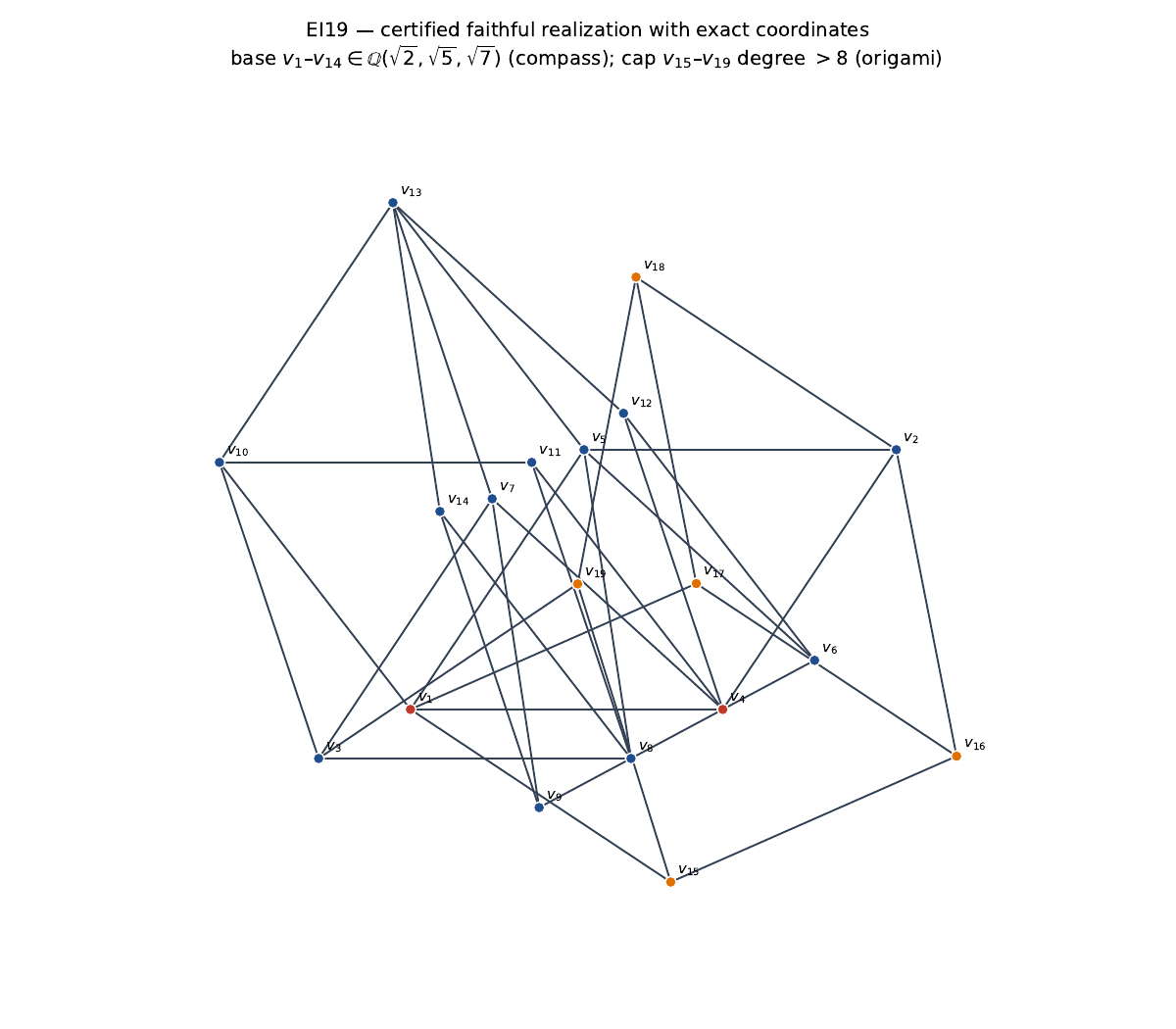}
\caption{Two steps with \emph{exact} coordinates (verified in SymPy): $v_7$
and $v_2$ by radical axis, with $M$, $L^2$ and $h$ explicit.}
\label{fig:exatas}
\end{figure}

\begin{figure}[htbp]
\centering
\includegraphics[width=\textwidth]{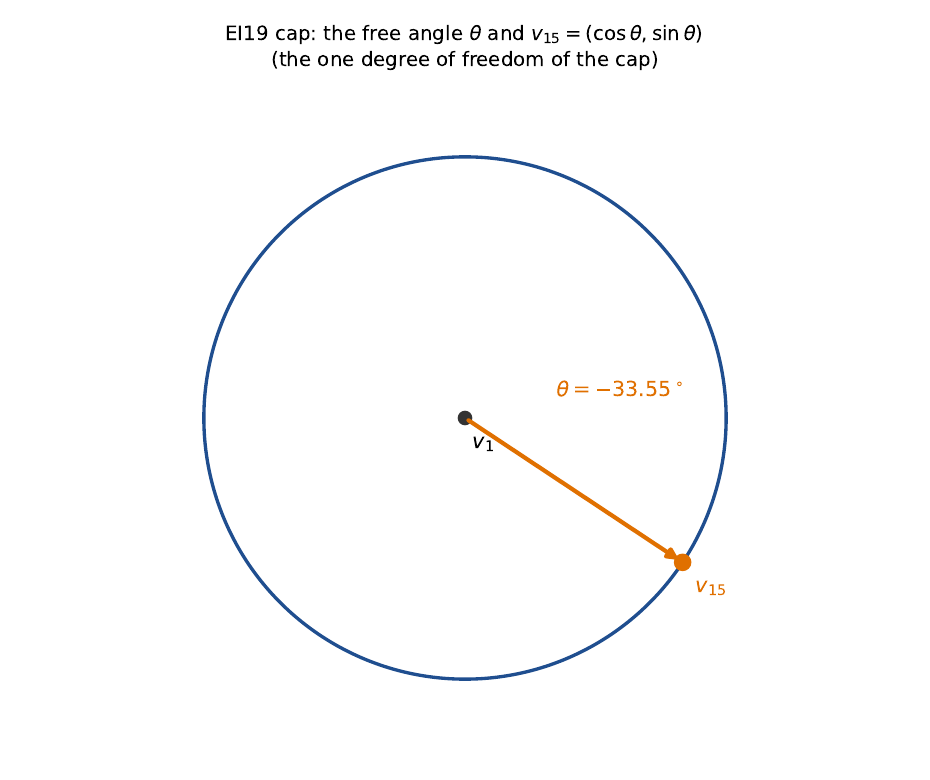}
\caption{Free angle $\theta=\arccos\tfrac59\approx56.25^\circ$ and the exact
coordinates of the $14$ base vertices, all in $\Q(\sqrt2,\sqrt5,\sqrt7)$.}
\label{fig:angcoord}
\end{figure}

The $14$ exact base coordinates (constants; they do \emph{not} depend on the
cap angle $\theta$) are:
\begin{equation}\label{eq:basecoords}
\renewcommand{\arraystretch}{1.35}
\begin{array}{llll}
\toprule
v & x_v & y_v & \text{field}\\ \midrule
v_1 & 0 & 0 & \Q\\
v_2 & \tfrac{14}{9} & \tfrac{2}{9}\sqrt{14} & \Q(\sqrt{14})\\
v_3 & -\tfrac{1}{9}\sqrt7 & -\tfrac{1}{9}\sqrt2 & \Q(\sqrt2,\sqrt7)\\
v_4 & 1 & 0 & \Q\\
v_5 & \tfrac{5}{9} & \tfrac{2}{9}\sqrt{14} & \Q(\sqrt{14})\\
v_6 & 1+\tfrac{1}{9}\sqrt7 & \tfrac{1}{9}\sqrt2 & \Q(\sqrt2,\sqrt7)\\
v_7 & \tfrac{5}{9}-\tfrac{1}{9}\sqrt7 & -\tfrac{1}{9}\sqrt2+\tfrac{2}{9}\sqrt{14} & \Q(\sqrt2,\sqrt7)\\
v_8 & 1-\tfrac{1}{9}\sqrt7 & -\tfrac{1}{9}\sqrt2 & \Q(\sqrt2,\sqrt7)\\
v_9 & 1-\tfrac{2}{9}\sqrt7 & -\tfrac{2}{9}\sqrt2 & \Q(\sqrt2,\sqrt7)\\
v_{10} & -\tfrac{1}{18}\sqrt7-\tfrac{1}{18}\sqrt{70} & -\tfrac{1}{18}\sqrt2+\tfrac{7}{18}\sqrt5 & \Q(\sqrt2,\sqrt5,\sqrt7)\\
v_{11} & 1-\tfrac{1}{18}\sqrt7-\tfrac{1}{18}\sqrt{70} & -\tfrac{1}{18}\sqrt2+\tfrac{7}{18}\sqrt5 & \Q(\sqrt2,\sqrt5,\sqrt7)\\
v_{12} & 1+\tfrac{1}{18}\sqrt7-\tfrac{1}{18}\sqrt{70} & \tfrac{1}{18}\sqrt2+\tfrac{7}{18}\sqrt5 & \Q(\sqrt2,\sqrt5,\sqrt7)\\
v_{13} & \tfrac{5}{9}-\tfrac{1}{18}\sqrt7-\tfrac{1}{18}\sqrt{70} & -\tfrac{1}{18}\sqrt2+\tfrac{7}{18}\sqrt5+\tfrac{2}{9}\sqrt{14} & \Q(\sqrt2,\sqrt5,\sqrt7)\\
v_{14} & 1-\tfrac{1}{6}\sqrt7-\tfrac{1}{18}\sqrt{70} & -\tfrac{1}{6}\sqrt2+\tfrac{7}{18}\sqrt5 & \Q(\sqrt2,\sqrt5,\sqrt7)\\
\bottomrule
\end{array}
\end{equation}

\subsection{From base to top: where the cap enters}
The cap $v_{15},\dots,v_{19}$ \textbf{hangs} from the base at three anchors,
$v_1,v_2,v_3$ (Figure~\ref{fig:ponte}). Everything in it is a function of a
\emph{single} free angle $\theta$ (the position of $v_{15}$ on the circle about
$v_1$).

\begin{figure}[htbp]
\centering
\includegraphics[width=.72\textwidth]{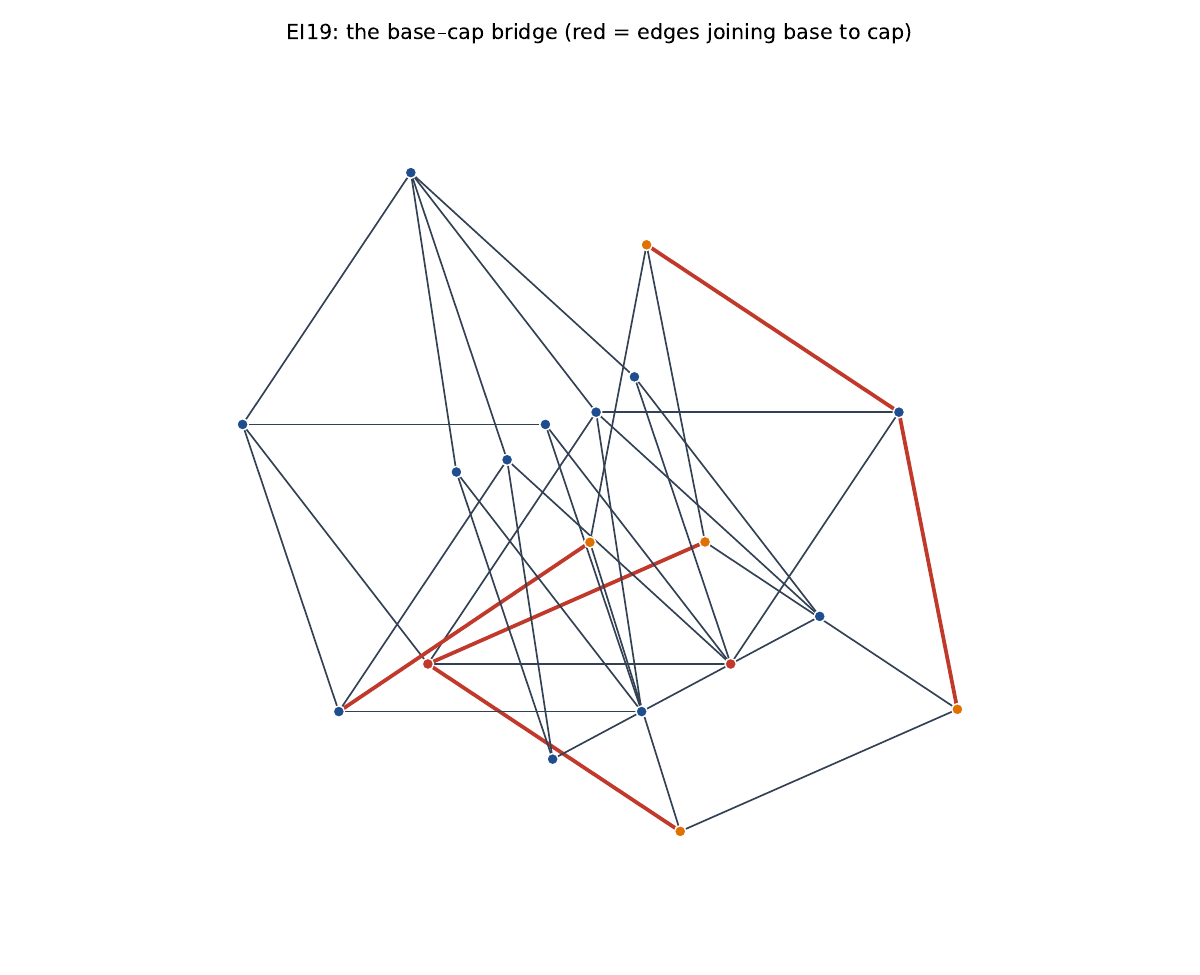}
\caption{The base~$\to$~cap bridge. The anchors $v_1,v_2,v_3$ (red) hold the
pentagon $v_{15}\cdots v_{19}$ (orange); the dashed edge $(18,19)$ is the
\emph{closure}.}
\label{fig:ponte}
\end{figure}

\subsection{Where the closure comes from}
As $v_{15}$ swings on the circle about $v_1$, a \emph{long arm}
$v_{16}\!\to\!v_{17}\!\to\!v_{18}$ and a \emph{short arm} $v_{19}$ swing
together. Writing the \emph{fold} (intersection of two unit circles) as
\[
\Phi_\varepsilon(A,B)=\frac{A+B}{2}+\varepsilon\sqrt{1-\tfrac14\lVert A-B\rVert^2}\,
\frac{R(B-A)}{\lVert B-A\rVert},\qquad R(x,y)=(-y,x),
\]
the five cap vertices become explicit functions of $\theta$:
\begin{equation}\label{eq:capfuncao}
\begin{aligned}
v_{15}(\theta)&=(\cos\theta,\ \sin\theta),\\
v_{16}(\theta)&=\Phi_{+}(v_2,v_{15}), & v_{17}(\theta)&=\Phi_{+}(v_1,v_{16}),\\
v_{18}(\theta)&=\Phi_{-}(v_2,v_{17}), & v_{19}(\theta)&=\Phi_{+}(v_3,v_{15}),
\end{aligned}
\end{equation}
with $v_1,v_2,v_3$ from table \eqref{eq:basecoords}. The \textbf{only remaining
edge}, $(18,19)$, equals $1$ only for special values of $\theta$: the
\textbf{closure} is the equation
\begin{equation}\label{eq:fechamento}
g(\theta)=\bigl\lVert v_{18}(\theta)-v_{19}(\theta)\bigr\rVert-1=0 .
\end{equation}
It is a mechanism we rotate until the two ends meet (Figure~\ref{fig:gtheta});
the geometric root is $\theta^*=-0.5856\ \text{rad}\ (=-33.5528^\circ)$.

\begin{observacao}[Why $\theta^*$ is negative]
The sign is a choice of \emph{orientation}. We measure $\theta$ from the
positive $x$-axis, counterclockwise. In the certified realization
$v_{15}=(\cos\theta^*,\sin\theta^*)=(0.8334,\,-0.5527)$ lies \emph{below} the
axis $v_1v_4$, i.e.\ $\sin\theta^*<0$, hence $\theta^*<0$ (a clockwise angle;
equivalently $\theta^*+2\pi\approx5.698$ rad). The cap folds into the lower
half-plane. The reflection $\theta\mapsto-\theta$ gives the \emph{mirror}
realization; and since the coordinate field depends on $\cos\theta$, which is
\emph{even} ($\cos(-\theta)=\cos\theta$), the degree-$12$ polynomial is the
\textbf{same} for both signs --- the sign fixes only the side, not the
arithmetic.
\end{observacao}

\begin{figure}[htbp]
\centering
\includegraphics[width=.8\textwidth]{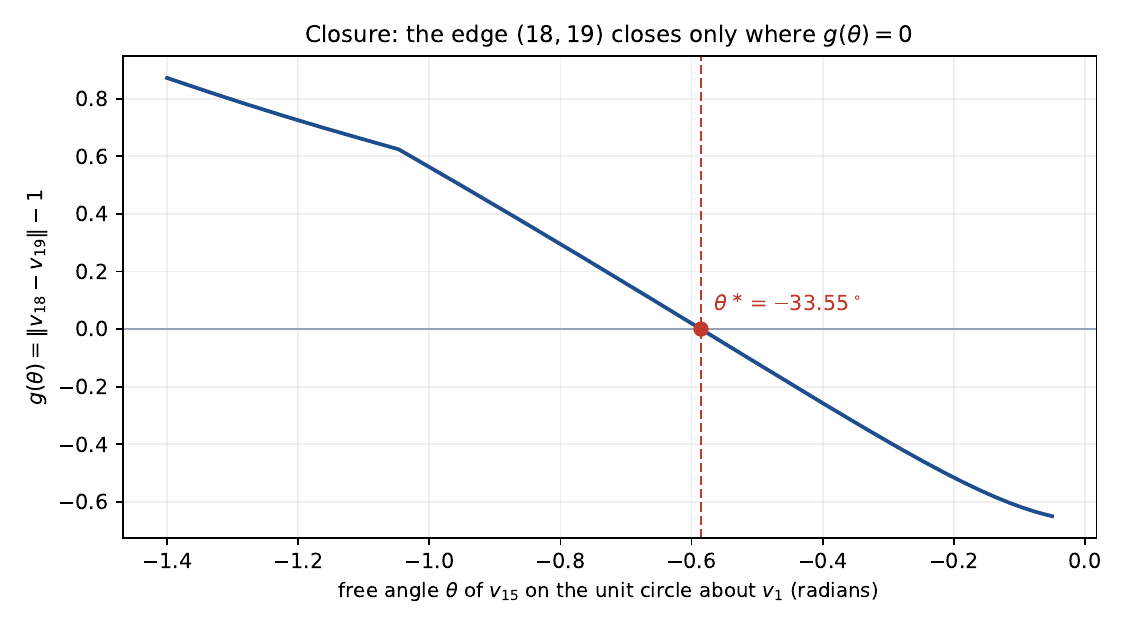}
\caption{The closure function $g(\theta)$ (horizontal axis in radians). The cap
``closes'' where $g(\theta)=0$; this is the \emph{only} condition that pins
$\theta$.}
\label{fig:gtheta}
\end{figure}

\begin{figure}[htbp]
\centering
\includegraphics[width=\textwidth]{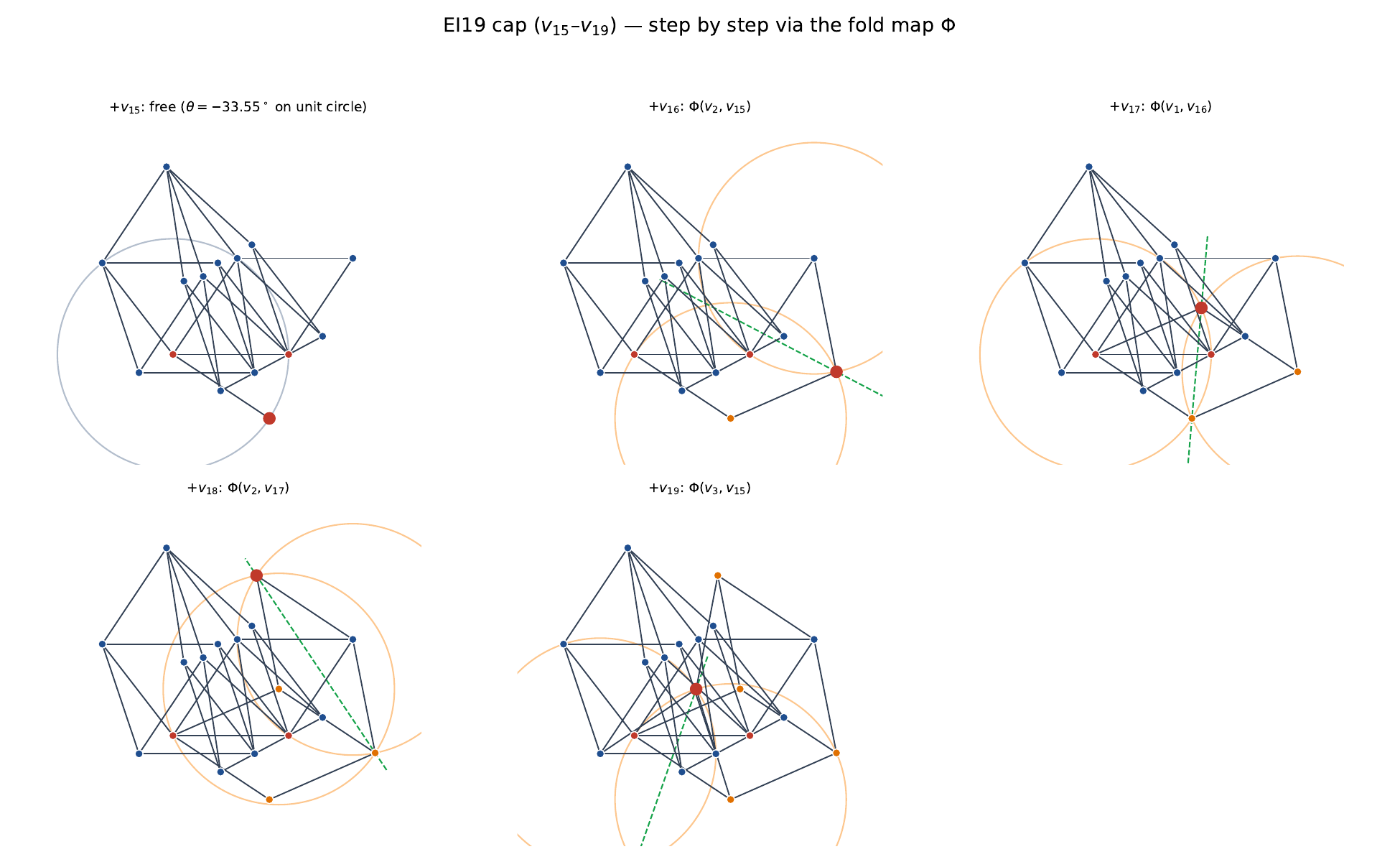}
\caption{Step-by-step vector construction of the \emph{cap}: $v_{15}$ (free
angle, degree $12$), $v_{16},v_{17},v_{18}$ (radical axis, H1), and $v_{19}$
(radical axis $+$ \textbf{closure} $(18,19)$, H2); finally the complete
$EI_{19}$.}
\label{fig:cappassos}
\end{figure}

\subsection{The orientation signs and the difficulty of the certificate}
\label{sec:sinais}

\subsubsection{What the sign $\varepsilon$ means}
Since $R$ rotates $90^\circ$ counterclockwise, $+\hat n$ points to the
\emph{left} of the direction $A\to B$. So the sign selects the side of the
perpendicular bisector:
\[
\varepsilon=+1\ \Longleftrightarrow\ \text{vertex to the \textbf{left} of }A\to B,\qquad
\varepsilon=-1\ \Longleftrightarrow\ \text{to the \textbf{right}}
\]
(Figure~\ref{fig:ramos}). It is the \emph{chirality} of each fold --- which of
the two circle intersections is taken. In the certified realization the
pattern is $(\,v_{16},v_{17},v_{18},v_{19}\,)=(+,+,-,+)$: three folds
``outward'' (left) and one reversal at $v_{18}$ (right), which folds the long
arm back to close $(18,19)$.

\begin{figure}[htbp]
\centering
\includegraphics[width=.62\textwidth]{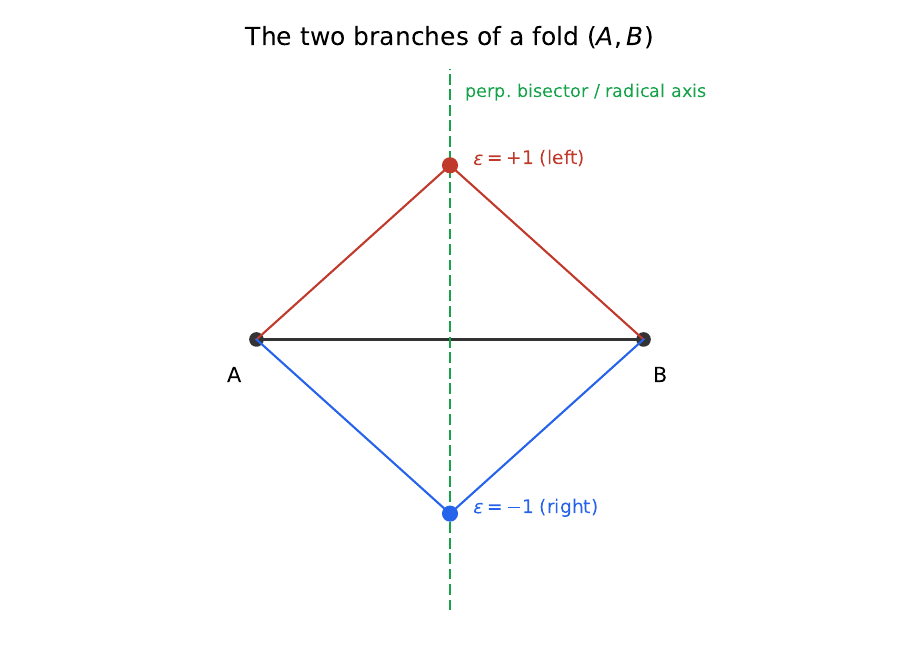}
\caption{The two branches of a fold $\Phi_\varepsilon(A,B)$: the unit circles
about $A$ and $B$ meet at $\Phi_{+}$ (left) and $\Phi_{-}$ (right), on the
perpendicular bisector, at a height $h=\sqrt{1-\tfrac14|A-B|^2}$ from $M$.}
\label{fig:ramos}
\end{figure}

\subsubsection{Why this makes the certificate hard}
The sign does \emph{not} decide whether the cap closes: a sweep shows that
\textbf{all $16$ sign combinations admit a real closure}. Each is a distinct
\emph{assembly mode} of the same mechanism (like a four-bar linkage
``elbow-up/elbow-down''). Two sources of difficulty emerge.

\smallskip\noindent\textbf{(i) Combinatorial explosion and degeneracy.}
A graph with $k$ vertices built by circle intersection has $2^{k}$ sign
patterns. For the cap ($k=4$) this is already $16$ modes; for the whole
$EI_{19}$, $2^{14}$. Most produce \emph{degenerate} or \emph{non-faithful}
configurations (coincident vertices or spurious unit distances). In practice a
$9000$-restart Newton search finds the faithful realization only $\sim 1$
time: the degenerate modes dominate. \emph{Finding} the correct branch is the
$\exists\R$-hard side (\S1.2).

\smallskip\noindent\textbf{(ii) The field is the invariant --- and needs
precision.} We confirmed that the degree-$12$ polynomial has \textbf{eight real
roots}
\[
\cos\theta\in\{0.0034,\ 0.1097,\ 0.8334,\ 0.8376,\ 0.8664,\ 0.8874,\ 0.8966,\ 0.9196\},
\]
and that \emph{several} faithful sign modes land on \emph{distinct} roots
(e.g.\ $0.8334$ --- the certified one ---, $0.8376$, $0.9196$). All are
\textbf{Galois conjugates} of the \emph{same} irreducible degree-$12$
polynomial: the coordinate field is a \emph{graph invariant}, independent of
which faithful realization the search returns. But exhibiting that invariant is
the \emph{arithmetic} side of the difficulty: the polynomial has height
$\sim 10^{10}$, and low precision returns \emph{spurious polynomials} (the
degree-$10$ pitfall); only with $\gtrsim 2600$ digits and \textsc{pslq} does
the true degree $12$ appear and persist. Verifying a realization is $O(n^2)$;
\emph{certifying the field} is the work.

\subsection{Henneberg construction: H1, H2 and the arithmetic difficulty}
\label{sec:henneberg}
Since $EI_{19}$ is isostatic ($35=2\cdot19-3$), it is a \emph{Laman graph} and
admits a \textbf{Henneberg construction}: starting from an edge ($K_2$), one
repeatedly applies one of two moves,
\begin{itemize}
\item \textbf{H1} (vertex addition): adds a \emph{degree-2} vertex, joined to
$2$ existing vertices;
\item \textbf{H2} (edge split): adds a \emph{degree-3} vertex, ``splitting'' an
existing edge.
\end{itemize}
Reducing $EI_{19}$ by the reverse moves brings it to $K_2$ with \textbf{$7$ H1
moves} and \textbf{$10$ H2 moves}, confirming (again) that it is Laman.

\subsubsection{The geometric correspondence is exact}
\[
\text{H1}\ \Longleftrightarrow\ \text{intersect two unit circles (radical axis)},
\qquad
\text{H2}\ \Longleftrightarrow\ \text{free angle } + \text{ closure}.
\]
In H1 the new vertex already has $2$ placed neighbours (two centres, one line
and one circle). In H2 it enters with a \emph{degree of freedom} (the free
angle) and a \emph{closure} edge locks it --- exactly the \emph{edge split}.
Table~\ref{tab:henneberg} classifies \textbf{each of the $19$ steps} in
construction order.

\begin{table}[htbp]
\centering\small\renewcommand{\arraystretch}{1.12}
\caption{Every step of the construction of $EI_{19}$ as Henneberg type 1 or
type 2. Degree $=$ number of already-placed neighbours. Free angles
($v_5,v_6,v_{15}$) open an H2 that a later closure completes.}
\label{tab:henneberg}
\begin{tabular}{rclcc}
\toprule
\# & $v$ & method & degree & Henneberg \\ \midrule
1 & $v_1$ & gauge & $0$ & --- \\
2 & $v_4$ & gauge & $1$ & --- \\
3 & $v_5$ & free angle $\theta_1$ & $1$ & opens H2 \\
4 & $v_2$ & radical axis $(v_4,v_5)$ & $2$ & H1 \\
5 & $v_6$ & free angle $\theta_2$ & $1$ & opens H2 \\
6 & $v_{12}$ & radical axis $(v_4,v_6)$ & $2$ & H1 \\
7 & $v_{13}$ & radical axis $(v_5,v_{12})$ & $2$ & H1 \\
8 & $v_7$ & radical axis $(v_4,v_{13})$ & $2$ & H1 \\
9 & $v_{10}$ & radical axis $(v_1,v_{13})$ & $2$ & H1 \\
10 & $v_3$ & radical axis $(v_7,v_{10})$ & $2$ & H1 \\
11 & $v_9$ & radical axis $(v_6,v_7)$ & $2$ & H1 \\
12 & $v_{11}$ & radical axis $(v_4,v_{10})$ & $2$ & H1 \\
13 & $v_8$ & radical axis $(v_3,v_5,v_{11})$ & $3$ & \textbf{H2} (closes $\theta_1$) \\
14 & $v_{14}$ & radical axis $(v_8,v_9,v_{13})$ & $3$ & \textbf{H2} (closes $\theta_2$) \\
15 & $v_{15}$ & free angle $\theta$ (cap) & $1$ & opens H2 \\
16 & $v_{16}$ & radical axis $(v_2,v_{15})$ & $2$ & H1 \\
17 & $v_{17}$ & radical axis $(v_1,v_{16})$ & $2$ & H1 \\
18 & $v_{18}$ & radical axis $(v_2,v_{17})$ & $2$ & H1 \\
19 & $v_{19}$ & radical axis $+$ closure $(18,19)$ & $3$ & \textbf{H2} (cap, degree 12) \\
\bottomrule
\end{tabular}
\end{table}

In total: $2$ gauge, $3$ free angles ($v_5,v_6,v_{15}$), $11$ \textbf{H1}
moves and $3$ \textbf{H2} closures ($v_8,v_{14},v_{19}$). The two base H2
closures resolve in square roots; the third (cap, $v_{19}$) is the degree-$12$
one --- origami.

\subsubsection{Henneberg-2 is where the Galois degree is decided}
This is the link to the arithmetic (the \emph{Laman-number conjecture}): the
\emph{tier} of the graph is read off the \textbf{Henneberg-2 closures}. Each
H2 is a closure; it is the degree (and Galois group) of that closure that fixes
the tier. In the base, the H2's resolve in square roots
($\Q(\sqrt2,\sqrt5,\sqrt7)$, compass). In the cap, the \emph{single} hard H2
--- the edge $(18,19)$ --- has degree $12=2^2\cdot3$: \emph{origami}. In
$EI_{17}$, the analogous H2 has a degree-$20$ resultant with group $S_{20}$:
\emph{exotic}, non-solvable. That is, \textbf{Henneberg-2 is literally where the
arithmetic difficulty lives}.

\subsection{All the preliminary equations (degree 2)}
With $v_1=(0,0)$, $v_2=(\tfrac{14}{9},\tfrac{2\sqrt{14}}{9})$,
$v_3=(-\tfrac{\sqrt7}{9},-\tfrac{\sqrt2}{9})$ fixed and unknowns
$c,s$ and $(x_{16},y_{16}),\dots,(x_{19},y_{19})$:
\begin{align}
\tag{$E_0$}& c^2+s^2-1=0 &&(v_{15}\text{ on the circle about }v_1)\\
\tag{$E_1$}& x_{16}^2-\tfrac{28}{9}x_{16}+y_{16}^2-\tfrac{4\sqrt{14}}{9}y_{16}+\tfrac{19}{9}=0\\
\tag{$E_2$}& x_{16}^2+y_{16}^2-2c\,x_{16}-2s\,y_{16}=0\\
\tag{$E_3$}& x_{17}^2+y_{17}^2-1=0\\
\tag{$E_4$}& x_{17}^2+y_{17}^2-2x_{16}x_{17}-2y_{16}y_{17}+x_{16}^2+y_{16}^2-1=0\\
\tag{$E_5$}& x_{18}^2-\tfrac{28}{9}x_{18}+y_{18}^2-\tfrac{4\sqrt{14}}{9}y_{18}+\tfrac{19}{9}=0\\
\tag{$E_6$}& x_{18}^2+y_{18}^2-2x_{17}x_{18}-2y_{17}y_{18}+x_{17}^2+y_{17}^2-1=0\\
\tag{$E_7$}& x_{19}^2+\tfrac{2\sqrt7}{9}x_{19}+y_{19}^2+\tfrac{2\sqrt2}{9}y_{19}-\tfrac{8}{9}=0\\
\tag{$E_8$}& x_{19}^2+y_{19}^2-2c\,x_{19}-2s\,y_{19}=0\\
\tag{$E_9$}& (x_{18}-x_{19})^2+(y_{18}-y_{19})^2-1=0 &&(\textbf{closure})
\end{align}

\begin{observacao}[The radical axis is a \emph{line}]
Subtracting the two circles of each vertex, the term $x^2+y^2$ \emph{cancels}
and a \textbf{line} remains (the perpendicular bisector of the centres, since
the radii are equal):
\begin{align*}
R_{16}=E_1-E_2:\ & \bigl(2c-\tfrac{28}{9}\bigr)x_{16}+\bigl(2s-\tfrac{4\sqrt{14}}{9}\bigr)y_{16}+\tfrac{19}{9}=0,\\
R_{17}=E_3-E_4:\ & 2x_{16}x_{17}+2y_{16}y_{17}-(x_{16}^2+y_{16}^2)=0,\\
R_{18}=E_5-E_6:\ & \bigl(2x_{17}-\tfrac{28}{9}\bigr)x_{18}+\bigl(2y_{17}-\tfrac{4\sqrt{14}}{9}\bigr)y_{18}+\tfrac{19}{9}=0,\\
R_{19}=E_7-E_8:\ & \bigl(2c+\tfrac{2\sqrt7}{9}\bigr)x_{19}+\bigl(2s+\tfrac{2\sqrt2}{9}\bigr)y_{19}-\tfrac{8}{9}=0.
\end{align*}
Each vertex is thus \textbf{line $\cap$ circle} $=$ two points: one
\emph{square root} per vertex.
\end{observacao}

\subsection{Where the cubic appears: at the closure}

\subsubsection{The only geometric locus is the circle}
Every unit-distance equation is a \textbf{circle} --- the locus of points at
distance $1$ from a centre. It is the \emph{only} conic that appears, and it
has degree~$2$. The line$\cap$circle construction gives, at each step, a
square root; the chain $v_{16}\!\to\!v_{17}\!\to\!v_{18}$ and $v_{19}$ is
merely a \emph{tower of square roots} over $K_0(c)$, and \emph{no} step
constrains $c$.

\subsubsection{The closure gives the degree $12$ \emph{directly} (irreducible)}
The \textbf{first} --- and only --- equation that pins $c$ is the closure
$E_9$. Eliminating the coordinates and clearing the branch roots, it produces
\emph{directly} the minimal polynomial $p(c)$, which is \textbf{irreducible of
degree $12=2^2\cdot3$} over $\Q$ (verified in SymPy). That is: $\cos\theta$ has
\emph{genuine} degree $12$ --- it is \textbf{not} a mere factorization, since
$p(c)$ does \emph{not} split over $\Q$.

\begin{observacao}[The only factorization is incidental]
When \emph{squaring} to eliminate the $\pm$ branch signs of the radical axes,
\emph{spurious} factors appear, corresponding to the \emph{other} branch
choices (other realizations and extraneous roots of the squaring). The
relevant factor --- the coordinate field of the chosen realization --- is the
irreducible degree-$12$ one. That factorization only \emph{selects the correct
branch}; it does not reduce the degree $12$.
\end{observacao}

\subsubsection{Where the ``$3$'' is: a subfield, not a factor}
Since $12$ is \emph{not} a power of $2$, the field $\Q(\cos\theta)$ has a tower
of subfields
\[
\Q\ \subset\ \Q(\sqrt7)\ \subset\ K_6\ \subset\ \Q(\cos\theta),
\qquad\text{of degrees } 2,\ 3,\ 2
\]
(Figure~\ref{fig:torre}). The ``$3$'' is \textbf{not} a factor of the
polynomial (which is irreducible): it is the \emph{cubic layer} of that tower,
generated by Cardano's cubic $\beta$ (irreducible over $\Q(\sqrt7)$). It is in
this precise sense that ``the degree-$3$ equation appears at the closure'':
the closure creates the degree-$12$ field, and the fact $3\mid12$ manifests as
this intermediate cubic step. The Galois group is $12T236$, solvable, of order
$2304=2^8\cdot3^2$.

\begin{figure}[htbp]
\centering
\includegraphics[width=.72\textwidth]{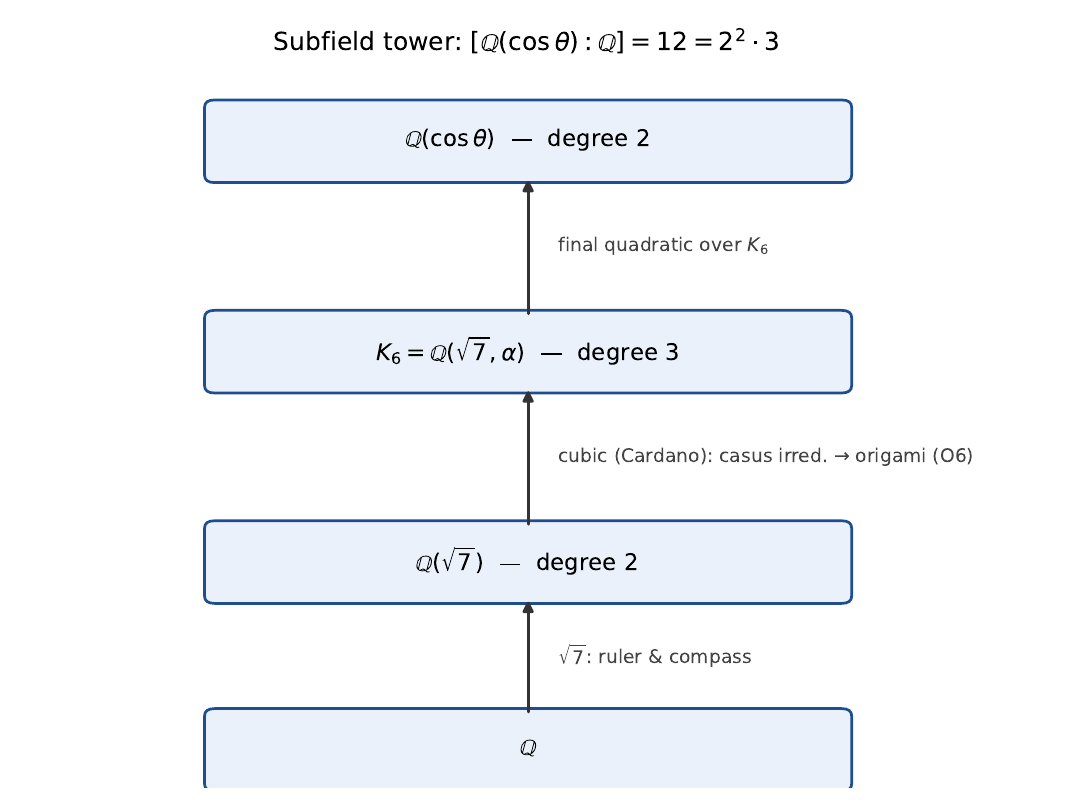}
\caption{The tower $\Q\subset\Q(\sqrt7)\subset K_6\subset\Q(\cos\theta)$ of
degrees $2,3,2$. The middle step (degree $3$) is Cardano's cubic.}
\label{fig:torre}
\end{figure}

\subsection{Cardano's formula and why $c=\cos\theta$ is real}

\subsubsection{The intermediate field $K_6$}
The cubic step $\Q(\sqrt7)\subset K_6$ is generated by the real root $\beta$ of
\begin{equation}\label{eq:beta}
\beta^3-(1+2\sqrt7)\,\beta^2-(37+5\sqrt7)\,\beta+(64+16\sqrt7)=0 .
\end{equation}
The field $K_6=\Q(\sqrt7,\beta)$ has degree $6=2\cdot3$ over $\Q$; concretely
$K_6=\Q[t]/(t^6-2t^5-101t^4+62t^3+1514t^2-3616t+2304)$. It is the compositum of
the quadratic layer $\Q(\sqrt7)$ (compass) with the cubic layer $\Q(\beta)$
(Cardano). Over $K_6$, finally, $\cos\theta$ satisfies one last quadratic,
closing the degree $12=6\cdot2$.

\begin{observacao}[Why only $\sqrt7$ in the tower --- not $\sqrt2,\sqrt5$]
The tower is one of subfields of $\Q(\cos\theta)$, the field of the
\emph{single} cap coordinate $\cos\theta$. The three radicals
$\sqrt2,\sqrt5,\sqrt7$ live in the \emph{base} coordinates (field
$\Q(\sqrt2,\sqrt5,\sqrt7)$), but on eliminating to get the minimal polynomial
of $\cos\theta$ over $\Q$ (degree $12$, integer coefficients), \emph{only}
$\sqrt7$ survives as a subfield. We verified this by factoring $p$: over
$\Q(\sqrt7)$ it splits into two degree-$6$ factors (so
$\sqrt7\in\Q(\cos\theta)$), whereas over $\Q(\sqrt2)$ and $\Q(\sqrt5)$ it
stays \emph{irreducible} of degree $12$ (so $\sqrt2,\sqrt5\notin\Q(\cos\theta)$).
Thus $\Q(\sqrt7)$ is the \emph{unique} quadratic subfield of $\Q(\cos\theta)$.
The radicals $\sqrt2,\sqrt5$ belong to the \emph{full} realization field --- the
compositum $\Q(\sqrt2,\sqrt5,\sqrt7,\cos\theta,\dots)$ of base and cap ---, not
to $\Q(\cos\theta)$. The tower displays only the subfield that carries the
origami obstruction; $\sqrt2,\sqrt5$ stay in the ``easy'' (compass) part of the
base.
\end{observacao}

\subsubsection{Cardano's formula}
For the depressed cubic $t^3+Pt+Q=0$ (obtained from \eqref{eq:beta} by the
shift $t=\beta-\tfrac{a_2}{3}$), Cardano gives
\[
t=\sqrt[3]{-\tfrac{Q}{2}+\sqrt{\Delta}}+\sqrt[3]{-\tfrac{Q}{2}-\sqrt{\Delta}},
\qquad
\boxed{\ \Delta=\frac{Q^2}{4}+\frac{P^3}{27}\ }
\]
where $\Delta$ is \emph{Cardano's discriminant}. Its sign decides the nature
of the roots:
\begin{itemize}
\item $\Delta>0$: one real root and two complex; the radicands are real.
\item $\Delta=0$: a multiple root.
\item $\Delta<0$: \textbf{three real roots}, yet the radicands
$-\tfrac{Q}{2}\pm\sqrt{\Delta}$ are \emph{complex} (since $\sqrt{\Delta}$ is
imaginary) --- the \textbf{\emph{casus irreducibilis}}.
\end{itemize}

\subsubsection{Why $\cos\theta$ is real, yet has no real radicals}
Here, with $P=-\tfrac{140}{3}-\tfrac{19\sqrt7}{3}$ and
$Q=\tfrac{595}{27}-\tfrac{403\sqrt7}{27}$,
\[
\Delta=-\frac{14126}{3}-\frac{10577\sqrt7}{6}\approx-9372.69<0 .
\]
Hence cubic \eqref{eq:beta} is in \emph{casus irreducibilis}: it has
\textbf{three real roots} ($\beta\approx10.195,\,-5.725,\,1.822$;
Figure~\ref{fig:cubica}). The crucial point:
\begin{quote}
$\cos\theta$ is a \emph{geometric} coordinate of a real realization ---
therefore it is \textbf{real}. It \emph{must} be one of the three real roots.
But Cardano's formula expresses it only through \emph{cube roots of complex
numbers} (because $\sqrt{\Delta}$ is imaginary). There is no expression of
$\cos\theta$ by \emph{real radicals}: the cube root is unavoidable.
\end{quote}
This is exactly the \textbf{angle-trisection} obstruction: possible by
\emph{origami} (the cubic Beloch fold $O6$), impossible by ruler and compass.
So the cap of $EI_{19}$ is origami --- not compass --- because of the ``$3$''
that the cubic introduces.

\begin{figure}[htbp]
\centering
\includegraphics[width=.8\textwidth]{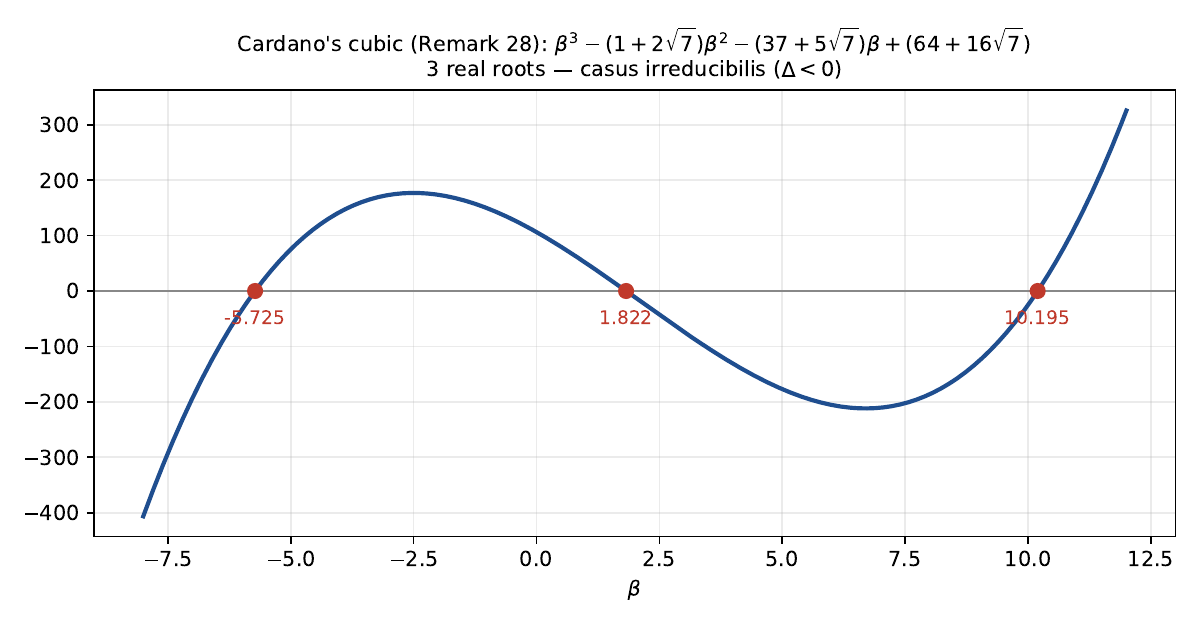}
\caption{Cardano's cubic \eqref{eq:beta} (Remark~\ref{obs:ei19casus}): three real roots,
$\Delta<0$. Real, yet with no real radicals.}
\label{fig:cubica}
\end{figure}

\subsubsection{The origami fold axioms (Huzita--Hatori)}
What a single fold can construct is given by the seven Huzita--Hatori axioms.
The first six yield a \emph{unique} determined fold; their algebraic ``power''
grows with the number of incidences required (Table~\ref{tab:axiomas}). The
frontier is axiom \textbf{$O6$} --- the Beloch fold ---, which solves
\emph{cubics}; $O1$--$O5$ and $O7$ stay at the quadratic level (equivalent to
ruler and compass).

\begin{table}[htbp]
\centering\small\renewcommand{\arraystretch}{1.25}
\caption{The Huzita--Hatori origami axioms and their algebraic degree. $p_i$
are points, $\ell_i$ lines.}
\label{tab:axiomas}
\begin{tabular}{clc}
\toprule
axiom & fold & degree \\ \midrule
$O1$ & the line through two points $p_1,p_2$ & $1$ \\
$O2$ & place $p_1$ onto $p_2$ (perpendicular bisector) & $1$ \\
$O3$ & place line $\ell_1$ onto $\ell_2$ (angle bisector) & $1$--$2$ \\
$O4$ & perpendicular to $\ell_1$ through $p_1$ & $1$ \\
$O5$ & place $p_1$ onto $\ell_1$ through $p_2$ (tangent to $1$ parabola) & $2$ \\
$\mathbf{O6}$ & place $p_1\!\to\!\ell_1$ \emph{and} $p_2\!\to\!\ell_2$ (common tangent to $2$ parabolas) & $\mathbf{3}$ \\
$O7$ & place $p_1$ onto $\ell_1$, perpendicular to $\ell_2$ & $2$ \\
\bottomrule
\end{tabular}
\end{table}

The single-fold origami-constructible numbers are exactly the \emph{Pierpont
numbers} --- those living in a tower of extensions of degrees $2$ and $3$
(prime support $\{2,3\}$). This is precisely the \emph{tier} of our
$\cos\theta$: degree $12=2^2\cdot3$.

\subsubsection{Lill's method and the Beloch fold}
How is a cube root produced \emph{geometrically}? By \textbf{Lill's method}
(Figure~\ref{fig:lill}). For the cubic $ax^3+bx^2+cx+d$, one draws a path of
four segments of lengths $a,b,c,d$, turning $90^\circ$ at each step, from $O$
to a terminal point $T$. A real root $x$ corresponds to a second path that
starts at $O$, reflects at $90^\circ$ off the lines of the segments and reaches
$T$; the launch angle $\varphi$ satisfies $\tan\varphi=-x$. Thus real roots
become a problem of \emph{drawing right angles}.

The bridge to origami is \textbf{Beloch's theorem}: the fold axiom $O6$
(Huzita--Hatori) --- \emph{folding so as to place two given points onto two
given lines simultaneously} --- constructs exactly Lill's path in a single
fold. Hence origami solves \emph{any} cubic, including the \emph{casus
irreducibilis} (three real roots without real radicals) and angle trisection
--- precisely the obstruction of our $\beta$. It is this $O6$ that constructs
$\cos\theta$: the cubic Beloch fold is the degree-$3$ step of the tower, beyond
ruler and compass.

\begin{figure}[htbp]
\centering
\includegraphics[width=.72\textwidth]{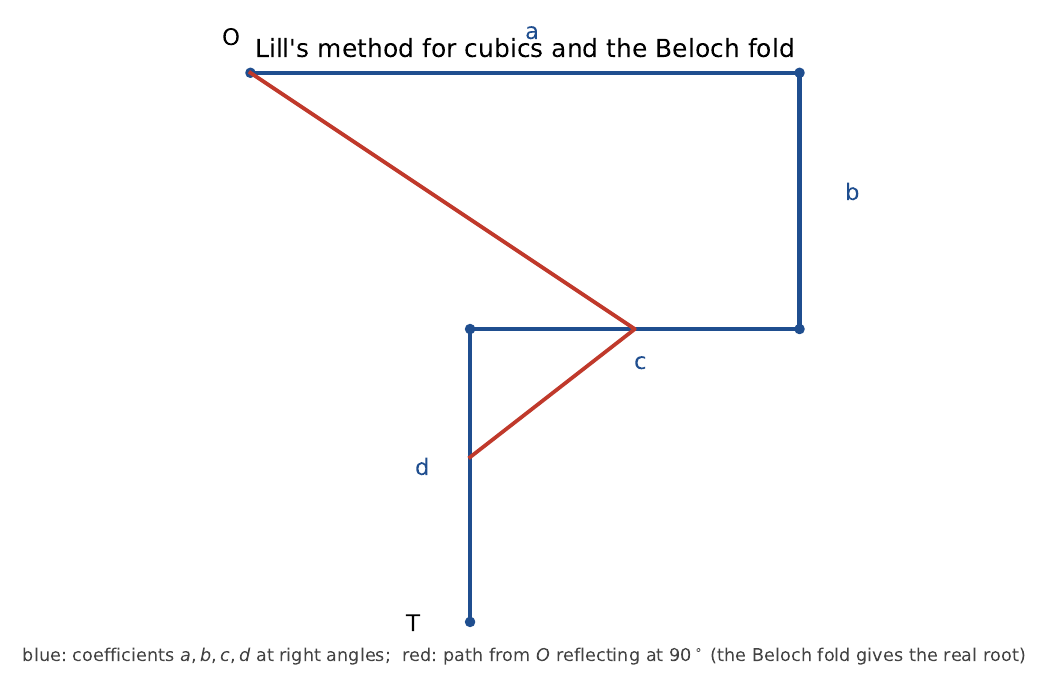}
\caption{Lill's method for the cubic (blue coefficient path $a,b,c,d$; red
``shooting'' path with $\tan\varphi=-x$) and the Beloch fold $O6$, which
realizes it in a single fold --- the geometric reason origami solves cubics.}
\label{fig:lill}
\end{figure}

\subsubsection{Angle trisection as the paradigm}
The canonical example of a cubic in \emph{casus irreducibilis} is
\textbf{trisection}. To trisect an angle $3\alpha$ (given $\cos 3\alpha$) one
needs $c=\cos\alpha$, which satisfies the triple-angle identity
\begin{equation}\label{eq:trissec}
4c^3-3c-\cos 3\alpha=0 ,
\end{equation}
a cubic. For $3\alpha=60^\circ$, for instance, $\cos 3\alpha=\tfrac12$ and the
cubic $8c^3-6c-1=0$ is irreducible over $\Q$ with three real roots
($\Delta<0$): the trisection of $60^\circ$ is \emph{impossible} by ruler and
compass (degree $3$ is not a power of $2$), but \emph{immediate} by origami ---
a single $O6$ fold solves \eqref{eq:trissec} via Lill. Our $\cos\theta$ is of
exactly the same type: the cubic $\beta$ \eqref{eq:beta} is the analogue of
\eqref{eq:trissec}, and the same $O6$ that trisects $60^\circ$ constructs the
cap of $EI_{19}$. Trisection is not a curiosity apart --- it is the
\emph{prototype} of the obstruction that classifies the graph as origami, not
compass.

\subsubsection{The solution is \emph{purely geometric}}
It is worth separating two things that are often conflated.

\smallskip\noindent\textbf{Constructing $EI_{19}$ is purely geometric.}
All vertices come from \emph{unit circles} and \emph{radical axes} (lines), and
the cap closes with \emph{a single fold} (Beloch $O6$). One need not write ---
nor even know --- the degree-$12$ polynomial to \emph{construct} the graph: it
suffices to intersect circles and fold. The Beloch fold solves the cubic
\emph{directly}, with no formula, no written radicals: the paper ``computes''
the cube root as it is folded.

\smallskip\noindent\textbf{The degree-$12$ polynomial is only the
\emph{certificate}.} It (with the Frobenius census) answers \emph{in which
field} the coordinates live and \emph{what the tier} is (origami, not
compass). It is \emph{arithmetic classification}, not construction. One can
fold $EI_{19}$ perfectly without ever exhibiting it.

\smallskip\noindent\textbf{Where Cardano's equation enters.} It is the
\emph{algebraic shadow} of the closure \eqref{eq:fechamento}. To ask ``can this
closure be done with ruler and compass alone?'' is to ask whether the cubic
$\beta$ \eqref{eq:beta} is solvable by real radicals; since $\Delta<0$
(\emph{casus irreducibilis}), the answer is \emph{no} --- origami is needed.
Geometrically, however, the closure is \emph{just a fold}: Cardano's cubic
never needs to be written to be \emph{solved}. It appears only when we
\emph{translate} the fold into algebra to certify the tier.

\subsection{Summary}
\begin{itemize}
\item Geometrically, \emph{everything} is degree $2$: circles (conics) and
lines (radical axes). Each vertex is line$\cap$circle $=$ one square root.
\item The \textbf{closure} $(18,19)$ is the only condition that pins $\theta$;
on eliminating it, one obtains \emph{directly} the \emph{irreducible}
degree-$12$ polynomial ($12=2^2\cdot3$; not a factorization).
\item The ``$3$'' is \textbf{Cardano's cubic} \eqref{eq:beta}; with $\Delta<0$
(casus irreducibilis) its three roots are real, yet $\cos\theta$ has no real
radicals $\Rightarrow$ \textbf{origami} (fold $O6$), not compass.
\item The field $K_6=\Q(\sqrt7,\beta)$ (degree $6$) is the compass $+$ Cardano
layer; one last quadratic over $K_6$ yields $\Q(\cos\theta)$, degree $12$,
group $12T236$.
\item In $EI_{17}$ the same circle mechanism produces, on elimination, degree
$20$ with $\Gal=S_{20}$ \emph{non-solvable}: not an isolated cubic, but a whole
group with no radical tower --- exotic.
\end{itemize}